\documentclass[11pt]{article}

\usepackage[margin=1in]{geometry}
\usepackage{amsmath,amssymb,amsfonts,mathrsfs,bm}
\usepackage{mathtools}
\usepackage{wrapfig}
\usepackage{cancel}
\usepackage[plain,noend]{algorithm2e}
\usepackage{amsthm}
\usepackage{natbib}
\usepackage[colorlinks=true,linkcolor=blue,citecolor=blue,urlcolor=blue]{hyperref}

\makeatletter
\renewcommand{\algocf@captiontext}[2]{#1\algocf@typo. \AlCapFnt{}#2}

\def\@algocf@capt@plain{top}
\renewcommand{\algocf@makecaption}[2]{%
  \addtolength{\hsize}{\algomargin}%
  \sbox\@tempboxa{\algocf@captiontext{#1}{#2}}%
  \ifdim\wd\@tempboxa >\hsize%
    \hskip .5\algomargin%
    \parbox[t]{\hsize}{\algocf@captiontext{#1}{#2}}%
  \else%
    \global\@minipagefalse%
    \hbox to\hsize{\box\@tempboxa}%
  \fi%
  \addtolength{\hsize}{-\algomargin}%
}
\makeatother

\theoremstyle{plain}
\newtheorem{theorem}{Theorem}
\newtheorem{lemma}{Lemma}
\newtheorem{corollary}{Corollary}
\theoremstyle{definition}
\newtheorem{remark}{Remark}

\title{An Entropy-Energy Identity for Predictive Kullback-Leibler Regret in Infinitely Divisible Location Models}
\author{K\={o}saku Takanashi\\RIKEN, Center for Advanced Intelligence Project\\\texttt{kosaku.takanashi@riken.jp}
\and Kenichiro McAlinn\\Department of Statistics, Operations, and Data Science\\Fox School of Business, Temple University\\\texttt{kenichiro.mcalinn@temple.edu}}
\date{}

\begin{document}
\maketitle

\begin{abstract}
We consider predictive density estimation under logarithmic score for $d$-dimensional infinitely divisible location models. Taking the formal Bayes predictive density under the Lebesgue prior as a benchmark, we study the Kullback--Leibler regret of competing Bayes predictive densities. Our main contribution is an exact entropy--energy identity: the integrated regret of a Bayes predictive density $\hat{p}^{\pi}$ under prior $\pi$ relative to the benchmark admits an exact representation as the Dirichlet-form energy of the square-rooted marginal distribution $\sqrt{M^{\pi}}$ for the symmetric Markov semigroup induced by the benchmark kernel. This converts regret comparisons into a potential-theoretic problem and yields a sharp recurrence/transience characterization of when the benchmark predictive density can or cannot be uniformly improved. We introduce an $\mathcal{A}$-harmonic class of improper priors-- defined through the generator $\mathcal{A}$ of the induced process-- and give explicit tail conditions-- an integral test on the induced marginal, equivalent to power-law prior decay in heavy-tailed models-- that guarantee admissibility of the resulting Bayes predictive density. We illustrate the theory with new results for several distributions.
\end{abstract}

\noindent\textbf{Keywords:} Predictive density estimation, Logarithmic score, Kullback--Leibler regret, Dirichlet forms, Symmetric Markov semigroups, $\mathcal{A}$-harmonic priors.

\section{Introduction\label{sec:intro}}

Predictive density estimation under logarithmic score-- equivalently,
expected Kullback--Leibler (KL) risk-- is a central decision problem
connecting prediction, compression, and Bayesian risk bounds \citep{clarke2002information,XieBarron2000}.
Recent work develops sharp minimax and dominance theory for KL predictive
risk in modern high-dimensional regimes \citep{XuZhou2011,Mukherjee-Johnstone_15},
and log predictive density has become a standard target for Bayesian
model evaluation and comparison \citep{VehtariGelmanGabry2017,FongHolmes2020}.
Given data $X=x$ generated from a parametric model indexed by $\theta$,
one seeks a predictive density $\hat{p}(\cdot|x)$ for a future observation
$Y$ that performs well under KL loss. 
While seemingly stylized, this one-sample decision paradigm is relevant. It directly models predictive tasks under extreme uncertainty, such as rare macroeconomic shocks, spatial localization from isolated sensors, and cold-start algorithmic decisions, and mathematically accommodates multi-sample scenarios whenever data can be aggregated into a single sufficient statistic.
A natural way to compare predictive
rules is through \emph{KL regret}: the excess log-score risk of a
candidate predictive density relative to a benchmark. Regret is attractive
because it is scale-free, interpretable in terms of expected log-likelihood,
and aligns closely with how predictive performance is assessed in
practice. However, outside a small number of classical settings, regret
is difficult to analyze directly, and it can be challenging to obtain
structural conclusions, such as whether a benchmark admits
uniform improvement.

This paper develops a simple principle for making KL regret analyzable
in infinitely divisible (ID) distributions: {translate entropy into energy}.
Our setting is the location model with unknown location $\theta\in\mathbb{R}^{d}$ and known Gaussian
variance and Lévy measure, together with a  benchmark predictive
density; i.e., the formal Bayes predictive density based on the (improper)
Lebesgue prior $\pi_{U}(\theta)\equiv1$, denoted $\hat{p}^{\pi_{U}}(y|x)$.
The key observation is that $\hat{p}^{\pi_{U}}(\cdot|x)$ is more
than a predictive rule. In the one sample and symmetric prior setting,
it also defines a symmetric Markov kernel, and hence a symmetric Markov
semigroup with an associated Dirichlet form capturing the ``energy''
of functions under the induced dynamics \citep{Fukushima-Oshima-Takeda_10}.
This probabilistic structure provides access to potential
theory for deciding whether energies can be made arbitrarily
small.

Our first main result is an exact \emph{entropy--energy identity}
(a generalized Brown identity): for a broad class of priors $\pi$,
the integrated KL regret of the Bayes predictive density $\hat{p}^{\pi}(\cdot|x)$
relative to the benchmark $\hat{p}^{\pi_{U}}(\cdot|x)$ can be represented
\emph{exactly} as a Dirichlet-form energy of a square-root transform
of the induced marginal $M^{\pi}$. This identity turns a decision-theoretic
regret comparison into an intrinsic potential-theoretic quantity attached
to the benchmark semigroup.

The identity yields the following statistical consequence. Dirichlet-form
theory provides sharp, model-agnostic structural dichotomies-- most
notably recurrence versus transience of the associated process--
that can be checked using standard tests. Exploiting these characterizations,
we obtain a recurrence criterion for admissibility of the benchmark
predictive density: recurrence of the induced symmetric Markov process
is equivalent to the absence of a uniformly better predictive density
(admissibility), whereas transience implies the existence of a uniform
improvement (inadmissibility). Thus, questions about uniform improvement
under KL risk reduce to recurrence/transience tests for the kernel
determined by $\hat{p}^{\pi_{U}}$. This reduction is particularly
useful in heavy-tailed and non-diffusive models, where direct KL-risk
analysis is typically intractable but recurrence classifications are
often available or can be verified via Dirichlet-form bounds.

Beyond admissibility of the benchmark, we also use the same framework
to construct \emph{admissible} formal Bayes solution based on improper
priors beyond $\pi_{U}$. Because the entropy--energy identity continues
to hold when the benchmark kernel is replaced by a general symmetric
Markov kernel, admissibility of a formal Bayes predictive density
$\hat{p}^{\rho}$, under prior $\rho$, can be established by proving recurrence of the
Markov process having $\hat{p}^{\rho}$ as its transition kernel.
Motivated by the role of harmonic functions in Brown's analysis of
the Gaussian location problem \citep{Brown_71}, we introduce an $\mathcal{A}$-harmonic
class of improper priors, where $\mathcal{A}$ denotes the generator
of the benchmark semigroup. For observation models with polynomial
tails $p(x|\theta)\asymp\|x-\theta\|^{-(d+\alpha)}$, we derive an
explicit sufficient condition for admissibility of $\hat{p}^{\rho}$:
an integral test expressed in terms of the induced marginal $M^{\rho}$.
When $\rho$ is spherically symmetric, this reduces to a explicit
tail requirement, $\rho(\theta)\asymp\|\theta\|^{-\beta}$ with $d-\alpha\le\beta\le d$,
thereby identifying a concrete family of improper priors that yield
admissible predictive densities in heavy-tailed models.

\textit{Related work.} Log-score predictive density estimation and
KL regret have been studied from information-theoretic and Bayesian
perspectives, including redundancy and asymptotic minimax analyses
and their links to reference and Jeffreys-type constructions \citep{clarke2002information,clarke1994jeffreys,rissanen2002fisher,XieBarron2000,Liang-Barron_04}.
In Gaussian models, KL predictive risk admits representations in terms
of quadratic-risk or Fisher-information quantities, which underpin
classical shrinkage and dominance results for predictive densities
\citep{Stein_55,Brown_71,EdGeorge-Liang-Xu_06,Brown-EdGeorge-Xu_08}.
Predictive improvements via shrinkage priors under KL loss have been
developed in a parallel line of work \citep{Komaki_01,Komaki_06,Komaki_21},
and sharp risk expansions for predictive densities have been investigated
in related settings \citep{Mukherjee-Johnstone_15}. From the probabilistic
side, Dirichlet forms provide analytic characterizations of recurrence
and transience for a wide class of symmetric Markov processes, including
non-local Lévy-type models \citep{Fukushima-Oshima-Takeda_10}. Our
contribution builds on these strands by giving an exact entropy--energy
representation for integrated KL regret and using recurrence/transience
and $\mathcal{A}$-harmonic constructions to obtain admissibility
criteria.




\section{Main Results\label{sec:main}}

This section develops the link between KL predictive risk and the potential theory of the benchmark predictive kernel. Section~\ref{sec:bayesrisk} first expresses the Bayes risk difference between a Bayes predictive density $\hat p^\pi$ and the benchmark $\hat p^{\pi_U}$ as an integrated conditional KL divergence. Theorem~\ref{prop:BayesRisk} then identifies this quantity exactly with the Dirichlet-form energy of $\sqrt{M^\pi}$ for the symmetric semigroup induced by $\hat p^{\pi_U}$.

Theorem~\ref{thm:characterization_Rec-Transi} is the statistical consequence of this identity: admissibility of the benchmark predictive density is reduced to recurrence/transience of the induced symmetric Markov process. Sections~\ref{sec:cauchy}, \ref{subsec:Related-rusults}, and \ref{subsec:A-harmonic} then apply this criterion to representative infinitely divisible models, relate it to the Gaussian case, and derive the sufficient condition for the $\mathcal A$-harmonic priors considered in Theorem~\ref{thm:A-harmonic}.

\subsection{Bayes risk difference and Blyth's method\label{sec:bayesrisk}}

Consider  $d$-dimensional infinitely divisible location
models:
\[
X=\theta+\varepsilon,\ \varepsilon\sim\textrm{ID}\left(A,\nu,\gamma\right)
\]
where $\theta$ is the location parameter, and $\left(A,\nu,\gamma\right)$
is a known Lévy-Khintchine triplet ($A$ is the Gaussian variance,
$\nu$ is the Lévy measure and $\gamma$ is the center). Write
the likelihood function as $p_{A,\nu,\gamma}\left(x\left|\theta\right.\right)$.
No moment assumptions are imposed. Based on observing a datum, $X=x$, we consider the problem
of obtaining a predictive density, $\hat{p}\left(y\left|x\right.\right)$,
for $Y$ that is close to $p_{A,\nu,\gamma}\left(y\left|\theta\right.\right)$.
We measure this closeness by the Kullback-Leibler (KL) loss, 
\[
L_{\mathsf{KL}}\left(\theta,\hat{p}\left(\cdot\left|x\right.\right)\right)=\mathsf{KL}\left(p_{A,\nu,\gamma}\left(\cdot\mid\theta\right)\middle\|\hat{p}\left(\cdot\mid x\right)\right)=\int_{\mathbb{R}^{d}}\log\frac{p_{A,\nu,\gamma}\left(y\left|\theta\right.\right)}{\hat{p}\left(y\left|x\right.\right)}p_{A,\nu,\gamma}\left(y\left|\theta\right.\right)dy,
\]
and evaluate $\hat{p}$ by its expected loss or risk function, $R_{\mathsf{KL}}\left(\theta,\hat{p}\right)=\int_{\mathbb{R}^{d}}L_{\mathsf{KL}}\left(\theta,\hat{p}\left(\cdot\left|x\right.\right)\right)p_{A,\nu,\gamma}\left(x\left|\theta\right.\right)dx$.
For the comparison of two procedures, we say that $\hat{p}^{1}$ dominates
$\hat{p}^{2}$ if $R_{\mathsf{KL}}\left(\theta,\hat{p}^{1}\right)\leq R_{\mathsf{KL}}\left(\theta,\hat{p}^{2}\right)$
for all $\theta$ and with strict inequality for some $\theta$. The
Bayes risk, $B_{\mathsf{KL}}\left(\pi,\hat{p}\right)=\int_{\mathbb{R}^{d}}R_{\mathsf{KL}}\left(\theta,\hat{p}\right)\pi\left(\theta\right)d\theta$,
is minimized by the predictive distribution $\hat{p}^{\pi}\left(y\left|x\right.\right)=\int_{\mathbb{R}^{d}}p_{A,\nu,\gamma}\left(y\left|\theta\right.\right)\pi\left(\theta\left|x\right.\right)d\theta$.
The predictive distribution dominates any plug-in estimate \citep{Aitchison_75}.
The best equivariant predictive density, with respect to the location
group transformation, is the Bayes predictive density under a uniform
prior, $\pi_{U}\left(\theta\right)\equiv1$, which has constant risk
\citep[see][]{Murray_77,Ng_80}. More precisely, one might refer to
$\hat{p}^{\pi_{U}}$ as a formal Bayes solution because $\pi_{U}$
is improper. \cite{Aitchison_75} showed that $\hat{p}^{\pi_{U}}\left(y\left|x\right.\right)$
dominates the plug-in predictive density $p_{A,\nu,\gamma}\left(y\left|\hat{\theta}_{\textrm{MLE}}\right.\right)$,
which simply substitutes the MLE, $\hat{\theta}_{\textrm{MLE}}=x$,
for $\theta$.

To examine the admissibility/inadmissibility of $\hat{p}^{\pi_{U}}$
under KL risk, Blyth's method is an effective strategy. The necessary
and sufficient condition for $\hat{p}$ to be admissible \citep{Stein_56, Farrell_66}
is for there to exist finite Borel measures, $\left\{ \pi_{i}\right\} ,i=1,2,\cdots$,
that satisfies 
\begin{align*}
\mathscr{A}:  \pi_{i}\left(B\right)\geq1; \quad  \mathscr{B}:  B_{\mathsf{KL}}\left(\pi_{i},\hat{p}\right)-B_{\mathsf{KL}}\left(\pi_{i},\hat{p}^{\pi_{i}}\right)\rightarrow0,\ \left\{ \pi_{i}\right\} \textrm{ has compact support}
\end{align*}
where $B=\left\{ \left.\theta\in\mathbb{R}^{d}\right|\left\Vert \theta\right\Vert \leq1\right\} $
is unit ball. Note that $\hat{p}^{\pi_{i}}$ denotes the Bayes estimator,
$\hat{p}^{\pi_{i}}=\int_{\mathbb{R}^{d}}p_{A,\nu,\gamma}\left(y\left|\theta\right.\right)\pi_{i}\left(\theta\left|x\right.\right)d\theta$,
under prior, $\pi_{i}$.

A straightforward calculation yields the following representation of the Bayes risk difference:

\begin{lemma}\label{BRD-invariant}

We have 
\begin{equation}
B_{\mathsf{KL}}\left(\pi,\hat{p}^{\pi_{U}}\right)-B_{\mathsf{KL}}\left(\pi,\hat{p}^{\pi}\right)=\int_{\mathbb{R}^{d}}\mathsf{KL}\left(\hat{p}^{\pi}\left(\cdot\mid x\right)\middle\|\hat{p}^{\pi_{U}}\left(\cdot\mid x\right)\right)M^{\pi}\left(x;A,\nu,\gamma\right)dx.\label{eq:BayesRiskDiffe}
\end{equation}
Here, $M^{\pi}\left(x;A,\nu,\gamma\right)$ is the marginal distribution
under $\pi$: $M^{\pi}\left(x;A,\nu,\gamma\right)=\int_{\mathbb{R}^{d}}p_{A,\nu,\gamma}\left(x\left|\theta\right.\right)\pi\left(\theta\right)d\theta$,
and $M^{\pi}\left(x;A,\nu,\gamma\right)$ is the invariant distribution
of $\hat{p}^{\pi}\left(y\left|x\right.\right)$: 
\begin{equation}
\int_{\mathbb{R}^{d}}\hat{p}^{\pi}\left(y\left|x\right.\right)M^{\pi}\left(y;A,\nu,\gamma\right)dy=M^{\pi}\left(x;A,\nu,\gamma\right).\label{eq:Marginal}
\end{equation}

\end{lemma}

\subsection{Equivalence between Bayes risk difference and Dirichlet form \label{sec:2.2}}

Our main contribution is to show that the Bayes risk difference eq.~\eqref{eq:BayesRiskDiffe}
in the above setup satisfies the following equality, which we call
generalized Brown's identity:

\begin{theorem} \label{prop:BayesRisk}

Let $M^{\pi}\left(x;A,\nu,\gamma\right)$ be the invariant distribution
of $\hat{p}^{\pi}$. The Bayes risk difference eq.~\eqref{eq:BayesRiskDiffe}
satisfies 
\begin{equation}
\int_{\mathbb{R}^{d}}\mathsf{KL}\left(\hat{p}^{\pi}\left(\cdot\mid x\right)\middle\|\hat{p}^{\pi_{U}}\left(\cdot\mid x\right)\right)M^{\pi}\left(x;A,\nu,\gamma\right)dx=\mathcal{E}\left(\sqrt{M^{\pi}\left(\cdot;A,\nu,\gamma\right)},\sqrt{M^{\pi}\left(\cdot;A,\nu,\gamma\right)}\right).\label{eq:BRD-Dirichlet}
\end{equation}
Here, $\mathcal{E}\left(f,f\right)$ and $f\in\mathcal{F}$ are the
values of the Dirichlet form of a Markov process, $\left\{ X_{t}\right\} _{t\geq0}$,
with $\hat{p}^{\pi_{U}}=\hat{p}^{\pi_{U}}\left(y\left|x\right.\right)$
as its transition probability, with initial value, $X_{0}=x$, that
takes value, $X_{1}=y$, at time, $t=1$, and the domain of its Dirichlet
form. Details of the Dirichlet form are defined below. See Appendix Section~\ref{subsec:Markov-Bayes}
regarding interpreting $\hat{p}^{\pi_{U}}$ as transition probability.

\end{theorem}

Here, $\mathcal{E}\left(\cdot,\cdot\right)$ is the Dirichlet form
of a Markov process, $\left\{ X_{t}\right\} _{t\geq0}$, with transition
probability, $\hat{p}^{\pi_{U}}\left(y\left|x\right.\right)$. The
value of $\mathcal{E}\left(\sqrt{M^{\pi}\left(\cdot;A,\nu,\gamma\right)},\sqrt{M^{\pi}\left(\cdot;A,\nu,\gamma\right)}\right)$
is equal to the quadratic variation of a Markov process, $\left\{ X_{t}\right\} $,
under the transformation of variable, $\sqrt{M^{\pi}\left(\cdot;A,\nu,\gamma\right)}$
(details regarding the connection between the predictive distribution
and the continuous-time Markov process are given in Appendix Section~\ref{subsec:Markov-Bayes}).
Intuitively, this is equal to the instantaneous variance: 
\[
\lim_{\varDelta t\rightarrow0}\frac{1}{\varDelta t}\mathbb{E}\left[\left(\sqrt{M^{\pi}\left(X_{t+\varDelta t};A,\nu,\gamma\right)}-\sqrt{M^{\pi}\left(X_{t};A,\nu,\gamma\right)}\right)^{2}\right],
\]
where $\mathbb{E}$ represents expectation of the Markov process,
$\left\{ X_{t}\right\} _{t\geq0}$, with respect to the transition
probability $\hat{p}^{\pi_{U}}$. $\mathcal{F}$ is the domain of
the infinitesimal generator, $-\mathcal{A}$, of the Markov process,
$\left\{ X_{t}\right\} $. 

Eq.~\eqref{eq:BRD-Dirichlet} is a generalization of eq.~(1.3.4)
in \cite{Brown_71} by extending the quadratic loss to KL loss, and
normal distribution to ID location models.

From the equality in eq.~\eqref{eq:BRD-Dirichlet}, we can connect
the statistical decision problem and the Markov process. Whether the
Markov process is recurrent or transient can be discerned by the Dirichlet
form, and is dependent on the existence of a sequence of functions
where the Dirichlet form becomes zero \citep{Fukushima-Oshima-Takeda_10}.
From this, we can correspond to Blyth's method and the discernment
of recurrence and transience of a Markov process through its Dirichlet
form. Hence, the admissibility or inadmissibility of $\hat{p}^{\pi_{U}}$
under KL risk has a corresponding relationship with the recurrence
or transience of a Markov process with transition probability $\hat{p}^{\pi_{U}}$.
Thus, we have,

\begin{theorem}\label{thm:characterization_Rec-Transi} Let $\mathsf{m}$
be a Lebesgue measure on $\mathbb{R}^{d}$, $\left\{ X_{t}\right\} $
be a Markov process with $\hat{p}^{\pi_{U}}\left(y\left|x\right.\right)$
as its transition probability density and $\left(\mathcal{E},\mathcal{F}\right)$
be the corresponding Dirichlet form on $L^{2}\left(\mathbb{R}^{d};\mathsf{m}\right)$. 
\begin{itemize}
\item A necessary and sufficient condition for the recurrence of a Markov
process with transition probability $\hat{p}^{\pi_{U}}$ is the
existence of a function sequence, $\left\{ f_{n}\right\} $, that
satisfies 
\[
\left\{ f_{n}\right\} \subset\mathcal{F},\quad\lim_{n\rightarrow\infty}f_{n}=1\left(\mathsf{m}\textrm{-a.e.}\right),\quad\lim_{n\rightarrow\infty}\mathcal{E}\left(f_{n},f_{n}\right)=0.
\]
Therefore, if $\left\{ X_{t}\right\} $ is recurrent then $\hat{p}^{\pi_{U}}\left(y\left|x\right.\right)$
is admissible. 
\item A necessary and sufficient condition for the transience of a Markov
process with transition probability $\hat{p}^{\pi_{U}}$ is that
there exists an $\mathsf{m}$-integrable function $g$ that is bounded
on $\mathbb{R}^{d}$ with $g>0,\ \mathsf{m}\textrm{-a.e.}$ and satisfies
\begin{alignat*}{2}
0<\int_{\mathbb{R}^{d}}\left|f\right|gd\mathsf{m}\leq & \mathcal{E}\left(\sqrt{f},\sqrt{f}\right), & \ \ \textrm{ for all }\sqrt{f}\in\mathcal{F}.
\end{alignat*}
Therefore, if $\left\{ X_{t}\right\} $ is transient then $\hat{p}^{\pi_{U}}\left(y\left|x\right.\right)$
is inadmissible. 
\item When $\hat{p}^{\pi_{U}}\left(y\left|x\right.\right)$ is the transition
probability, the Markov process $\left\{ X_{t}\right\} $ is irreducible,
thus either of the above two will
occur. 
\end{itemize}
\end{theorem}

Recurrence and transience of the induced process admit sharp, model-agnostic
characterizations in terms of the Dirichlet form.

\subsection{Application to various distributions\label{sec:cauchy}}

Given the above result, we can immediately apply the theorem to various
infinitely divisible (ID) distributions. 
Estimating the predictive distribution
under KL risk with the uniform prior yields a predictive process corresponding
to the symmetrized triplet $\textrm{ID}\left(2A,\tilde{\nu},0\right)$,
where $\tilde{\nu}(dx)=\nu(dx)+\nu(-dx)$. The admissibility is therefore
completely determined by the recurrence of this symmetrized pure-jump
Lévy process.

For the $d$-dimensional Cauchy process, generated by $\textrm{ID}\left(0,\nu_{\mathcal{C}},0\right)$
where $\nu_{\mathcal{C}}(dx)=c\|x\|^{-(d+1)}dx$, it is known that
the recurrence and transience switches between $d=1$ and $d\geq2$
\citep{Sato_99}. This leads to the following analog of \cite{Brown_71}
for the Cauchy distribution:

\begin{corollary}\label{cor:Cauchy} Consider a $d$-dimensional
Cauchy distribution parameterized as $\textrm{ID}\left(0,\nu_{\mathcal{C}},\gamma+\theta\right)$
with an unknown location $\theta$ and a known scale $c$. For the
problem of estimating the predictive distribution under KL risk, upon
observing $X=x$:

When $d=1$, the uniform prior Bayes predictive distribution is admissible.

When $d\geq2$, the uniform prior Bayes predictive distribution is
inadmissible. \end{corollary}

Furthermore, for distributions with finite variance, the recurrence
of the symmetrized Lévy process exhibits a dimensional phase transition
identical to that of the normal distribution. A pure-jump Lévy process
with a finite non-degenerate covariance matrix behaves asymptotically
like a Brownian motion; thus, it is recurrent in dimensions $d\leq2$
and transient in $d\geq3$.

\begin{corollary}\label{cor:FiniteVariance} Consider a $d$-dimensional
location model $X\sim\textrm{ID}\left(0,\nu,\gamma+\theta\right)$
with an unknown location $\theta\in\mathbb{R}^{d}$. If the underlying
Lévy measure $\nu$ has a finite second moment (i.e., $\int_{\|x\|>1}\|x\|^{2}\nu(dx)<\infty$)
and a non-degenerate covariance matrix, the uniform prior Bayes predictive
distribution is: admissible when $d\leq2$, and inadmissible when
$d\geq3$.

This immediately establishes the admissibility transition for $d$-dimensional
extensions (e.g., vectors with independent components) of the following
asymmetric base distributions, characterized by their $1$-dimensional
base Lévy measures $\nu_{0}$: 
\begin{itemize}
\item Poisson distribution: $\nu_{0}(dx)=\lambda\delta_{1}(dx)$. The symmetrized
predictive distribution corresponds to the symmetric Skellam process. 
\item Exponential distribution: $\nu_{0}(dx)=x^{-1}e^{-\lambda x}\mathbf{1}_{x>0}dx$.
The symmetrized predictive distribution corresponds to the symmetric
Laplace process. 
\item Gamma distribution: $\nu_{0}(dx)=kx^{-1}e^{-x/c}\mathbf{1}_{x>0}dx$.
The symmetrized predictive distribution corresponds to the symmetric
Variance-Gamma (VG) process. 
\item Inverse Gaussian distribution: $\nu_{0}(dx)=cx^{-3/2}e^{-\lambda x}\mathbf{1}_{x>0}dx$.
The symmetrized predictive distribution corresponds to the symmetric
Normal-Inverse Gaussian (NIG) process. 
\item Gumbel distribution: $\nu_{0}(dx)=x^{-1}(1-e^{-x})^{-1}e^{-x}\mathbf{1}_{x>0}dx$.
The symmetrized predictive distribution corresponds to the symmetric
Logistic process. 
\end{itemize}
\end{corollary}

Conversely, if the variance is infinite, admissibility dictates a
strict phase transition depending on the tail index $\alpha$.

\begin{corollary}\label{cor:Stable} Consider a $1$-dimensional
location model $X\sim\textrm{ID}\left(0,\nu,\gamma+\theta\right)$
with an unknown location $\theta\in\mathbb{R}$, where the base distribution
exhibits a general $\alpha$-stable distribution (either symmetric
or asymmetric) with Lévy measure $\nu(dx)=\left(c_{1}\mathbf{1}_{x>0}+c_{2}\mathbf{1}_{x<0}\right)\left|x\right|^{-(1+\alpha)}dx$.
The uniform prior Bayes predictive distribution yields a 
symmetric $\alpha$-stable process with $\tilde{\nu}\left(dx\right)=\left(c_{1}+c_{2}\right)\left|x\right|^{-\left(1+\alpha\right)}dx$.
When $1\leq\alpha<2$, the symmetrized process is recurrent (where
$\alpha=1$ recovers the Cauchy case), making the uniform prior Bayes
predictive distribution admissible. When $0<\alpha<1$, the process
is transient, rendering the predictive distribution inadmissible regardless
of the skewness of the base distribution.

\end{corollary}

\begin{remark}[Distributions with Intractable Lévy Measures] An advantage of the criterion is that explicitly deriving the
exact closed form of the Lévy measure $\nu(dx)$ is not required to
determine admissibility. Many well-known distributions are infinitely
divisible (often belonging to the class of generalized Gamma convolutions)
but lack simple elementary representations for their Lévy measures.
For instance: 
\begin{itemize}
\item Distributions such as the Weibull distribution (with shape parameter
$k\leq1$), the Student's $t$-distribution (with degrees of freedom
$>2$), and the $F$-distribution (with appropriate degrees of freedom)
are ID and have finite variances. By Corollary \ref{cor:FiniteVariance},
their corresponding uniform prior Bayes predictive distributions are
admissible in dimensions $d\leq2$ and inadmissible in $d\geq3$. 
\item Heavy-tailed ID distributions such as the Pareto, half-Cauchy, and
Student's $t$-distribution (with degrees of freedom $\leq2$) have
infinite variances. Their admissibility can still be individually
verified by examining their tail indices, analogous to the phase transition in Corollary \ref{cor:Stable}. 
\end{itemize}
Thus, the recurrence of the symmetrized predictive process offers
a criterion that bypasses the need for explicit Lévy-Khintchine
representations. \end{remark}

\subsection{Reduction to the Gaussian Case\label{subsec:Related-rusults}}

Regarding the prediction problem of evaluating the Bayesian predictive
density using Kullback-Leibler loss, we have results for Gaussian
with unknown mean and known variance \cite{Komaki_01,EdGeorge-Liang-Xu_06,Brown-EdGeorge-Xu_08,Mukherjee-Johnstone_15}.
The result of \cite{Brown-EdGeorge-Xu_08} states that when $X\sim\mathcal{N}_{d}\left(\mu,v_{x}I\right)$
and $Y\sim\mathcal{N}_{d}\left(\mu,v_{y}I\right)$, the Bayes risk
difference between the predictive density, $\hat{p}^{\pi_{U}}\left(\left.y\right|x\right)$,
under uniform prior, $\pi_{U}\equiv1$, and the predictive density,
$\hat{p}^{\pi}\left(\left.y\right|x\right)$, under proper prior,
$\pi$ is
\begin{equation}
B_{\mathsf{KL}}\left(\mu,\hat{p}^{\pi_{U}}\right)-B_{\mathsf{KL}}\left(\mu,\hat{p}^{\pi}\right)=2\int_{v_{w}}^{v_{x}}\left[\int_{\mathbb{R}^{d}}\left\Vert \nabla_{z}\sqrt{M^{\pi}\left(z;vI,0\right)}\right\Vert ^{2}dz\right]dv.\label{eq:BGX-08_2}
\end{equation}
where the inside of the integral of the right hand side of eq.~\eqref{eq:BGX-08_2}
(regarding $v$) is equivalent to \citep[eq.~(1.3.4)]{Brown_71}.
In \cite{Brown-EdGeorge-Xu_08}, the prediction problem is reduced
to the normal mean problem, and the admissibility/inadmissibility
is shown using the results from \cite{Brown_71}. Although eqs.~\eqref{eq:BRD-Dirichlet}
and \eqref{eq:BGX-08_2} do not, initially, look equivalent, when
$v_{x}=v_{y}$, we have $v_{w}=\frac{1}{2}v_{x}$, making them equivalent.

\begin{lemma} \label{lem:BGX-Cor1} When $X\sim\mathcal{N}_{d}\left(\mu,v_{x}I\right)$,
$Y\sim\mathcal{N}_{d}\left(\mu,v_{y}I\right)$, and $v_{x}=v_{y}$,
we have $v_{w}=\frac{1}{2}v_{x}$, and the following holds: 
\[
B_{\mathsf{KL}}\left(\mu,\hat{p}^{\pi_{U}}\right)-B_{\mathsf{KL}}\left(\mu,\hat{p}^{\pi}\right)=\mathcal{E}_{\textrm{BM}}^{2v_{x}}\left(\sqrt{M^{\pi}\left(\cdot;v_{x}I,0\right)},\sqrt{M^{\pi}\left(\cdot;v_{x}I,0\right)}\right).
\]
The Dirichlet form, $\mathcal{E}_{\textrm{BM}}^{2v_{x}}\left(\cdot,\cdot\right)$,
corresponds to a diffusion process $Z_{t}=\sqrt{2v_{x}}W_{t}$ where
$W_{t}$ is a standard Brownian motion.
\end{lemma}

From this, we can see that the results for normal distributions in
\cite{Brown-EdGeorge-Xu_08} also hold within the framework of this
paper. The converse is also true, in that our results hold under the
setting in \cite{Brown-EdGeorge-Xu_08} when $v_{x}=v_{y}$.

\subsection{$\mathcal{A}$-harmonic prior\label{subsec:A-harmonic}}

Theorem~\ref{thm:characterization_Rec-Transi} links admissibility
of the benchmark predictive density $\hat{p}^{\pi_{U}}$ to recurrence
of the symmetric Markov process it induces. The entropy--energy identity
in Theorem~\ref{prop:BayesRisk} is in fact kernel-agnostic: if the
baseline predictive density is replaced by a general symmetric Markov
kernel $p(y|x)$, the same identity holds with the corresponding Dirichlet
form. Hence, for a (possibly improper) symmetric prior $\rho$, admissibility
of the formal Bayes predictive density 
\[
\hat{p}^{\rho}(y|x)=\frac{1}{M^{\rho}\left(x;A,\nu,\gamma\right) }\int_{\mathbb{R}^{d}}p_{A,\nu,\gamma}(y|\theta)p_{A,\nu,\gamma}(x|\theta)\rho(\theta)\,d\theta
\]
is guaranteed once the Markov process with transition density $\hat{p}^{\rho}$
is recurrent.

In the Gaussian location model the generator is the Laplacian, and
the classical harmonic/Stein prior arises from harmonicity at the
recurrence boundary. For heavy-tailed and non-local models the relevant
generator $\mathcal{A}$ is typically an integro-differential operator,
and we call priors comparable to positive $\mathcal{A}$-harmonic
(potential-theoretic) profiles as simply \emph{$\mathcal{A}$-harmonic}. Theorem~\ref{thm:A-harmonic}
gives a sufficient condition-- an integral test in terms of
$M^{\rho}$--, which under spherical symmetry, reduces to an explicit
power-law range for $\rho$.

\begin{theorem} \label{thm:A-harmonic}Suppose that in a $d$-dimensional infinitely divisible location model, the tail decay of the
observation model satisfies 
\[
p_{A,\nu,\gamma}(x|\theta)\sim\frac{1}{\|x-\theta\|^{d+\alpha}}\quad(\|x-\theta\|\to\infty).
\]
Then, a sufficient condition for the Bayes predictive density $\hat{p}^{\rho}(y|x)$
to be admissible under the Kullback-Leibler loss (and similar loss
functions) is that the marginal likelihood $M^{\rho}\left(r;A,\nu,\gamma\right)$ with $r=\|x\|$
satisfies 
\[
\int_{1}^{\infty}\frac{1}{M^{\rho}\left(r;A,\nu,\gamma\right)r^{d-\alpha+1}}dr=\infty.
\]
In particular, when the improper prior $\rho$ is spherically symmetric,
this sufficient condition is equivalent to the existence of some $\beta$
($d-\alpha\le\beta\le d$) such that 
\[
\rho(\theta)\sim\frac{1}{\|\theta\|^{\beta}}\quad(\|\theta\|\to\infty).
\]
\end{theorem}

\newpage{}

\appendix
\centerline{\textbf{\Large{}{}{}{}{}{}{}{}{}Appendix}{\Large{}{}{}{}{}{}{}{}{}}}

\setcounter{equation}{0} 
\global\long\def\theequation{S\arabic{equation}}%
\setcounter{page}{1}



\section{\label{subsec:Markov-Bayes}Bayesian predictive distributions as
Markov transition probabilities}

In this section, we establish a correspondence between the predictive distribution, $\hat{p}^{\pi_{U}}(y|x)$,
and the continuous-time Markov process, $\{X_{t}\}_{t\ge0}$. Thus,
we will derive the continuous-time Markov process from the predictive
distribution, $\hat{p}^{\pi_{U}}(y|x)$, and connect the known results
from Markov processes to the statistical decision problem. Specifically,
we define the Dirichlet form that corresponds to the predictive distribution,
$\hat{p}^{\pi_{U}}(y|x)$, and the concept of transience/recurrence
of Markov processes. For a more detailed explanation regarding Markov
processes, see \cite{Fukushima-Oshima-Takeda_10}. Let the transition
probability of a continuous-time Markov process, $\{X_{t}\}_{t\ge0}$,
be $p_{t}(y|x)$. This can be interpreted as the density function
of the probability that a Markov process, $\{X_{t}\}$, with initial
value, $X_{0}=x$, takes the value, $X_{t}=y$, at time, $t$.

Let us reinterpret the estimation problem. Consider a location model
$X=\theta+\varepsilon$ where the noise $\varepsilon$ follows an
infinitely divisible (ID) distribution $p(x|\theta)=f(x-\theta)$.
First, under the uniform prior $\pi_{U}(\theta)\propto1$, the predictive
distribution evaluated under the KL risk becomes the convolution:
\[
\hat{p}^{\pi_{U}}(y|x)=\int_{\mathbb{R}^{d}}f(y-\theta)f(x-\theta)d\theta=\int_{\mathbb{R}^{d}}f(\eta)f(\eta-(y-x))d\eta.
\]
This convolution effectively ``symmetrizes" the underlying noise.
The resulting $\hat{p}^{\pi_{U}}(y|x)$ depends only on $y-x$ and
is symmetric. Because the base distribution is ID, this
symmetrized distribution uniquely generates a symmetric Lévy process
(a continuous-time Markov process with stationary and independent
increments). Thus, $\hat{p}^{\pi_{U}}(y|x)$ serves as the transition
density $p_{1}(y|x)$ of this symmetric Lévy process. 

When associating the transition density function, $p_{t}(y|x)$, of
the continuous-time stochastic process, $\{X_{t}\}_{t\ge0}$, to the
predictive distribution, $\hat{p}^{\pi_{U}}(y|x)$, we assumed that
the transition density function, $p_{t}(y|x)$, at time, $t=1$, is
equal to the predictive distribution, $\hat{p}^{\pi_{U}}(y|x)$. Since
the symmetrized distribution under the uniform prior is infinitely
divisible, it naturally embeds into a continuous convolution semigroup
$\{\mu_{t}\}_{t\ge0}$. The transition density function, $p_{t}(y|x)$,
can take any positive real value at time $t$. For example, if the
original ID distribution is characterized by the Lévy-Khintchine triplet
$\mathrm{ID}(A,\nu,\gamma)$, the transition probability $p_{t}(y|x)$
at time $t$ exactly corresponds to the probability distribution generated
by the scaled symmetrized triplet $\mathrm{ID}(t\cdot2A,t\cdot\tilde{\nu},0)$,
where $\tilde{\nu}(dx)=\nu(dx)+\nu(-dx)$. Therefore, constructing
a continuous-time transition probability function from the uniform
Bayes predictive distribution is straightforward.

Furthermore, this  connection to symmetric Markov processes
extends beyond the uniform prior. If we employ a general symmetric
prior $\pi(\theta)=\pi(-\theta)$ (such as the harmonic prior, which
plays a crucial role in establishing admissibility in higher dimensions),
the predictive distribution satisfies the detailed balance condition
with respect to the marginal likelihood $M^{\pi}(x)=\int f(x-\theta)\pi(\theta)d\theta$:
\[
M^{\pi}(x)\hat{p}^{\pi}(y|x)=\int_{\mathbb{R}^{d}}f(y-\theta)f(x-\theta)\pi(\theta)d\theta=M^{\pi}(y)\hat{p}^{\pi}(x|y).
\]
This detailed balance equation guarantees that $\hat{p}^{\pi}(y|x)$
induces a reversible (symmetric) Markov process with respect to the
measure $M^{\pi}(x)$. This symmetric structure allows us to utilize
the Dirichlet form theory to evaluate transience and recurrence, laying
the mathematical foundation for proving the admissibility of the predictive
distribution under symmetric priors such as the harmonic prior. 

For a general symmetric prior (or proper prior) $\pi$, the analytic
expression of the continuous-time Markovian semigroup is not as straightforward
as $\hat{p}^{\pi_{U}}$, and the expression of the cylindrical measure
is difficult to obtain as well. However, as shown by the detailed
balance equation above, it is known that the marginal likelihood $M^{\pi}(x)$
acts as the stationary probability measure and induces a stationary
Markov process. Here, the initial value is $x$, and the stationary
distribution follows $M^{\pi}(x)$. Denote the cylindrical measure
as $\mathbb{Q}_{x}$. Because making $M^{\pi}(x)$ a stationary distribution
is not sufficient to uniquely determine the stationary Markov process
pathwise, there are multiple choices regarding the cylindrical measure,
$\mathbb{Q}_{x}$. The least we can say is that the Bayes predictive
distribution, $\hat{p}^{\pi}(y|x)$, is a marginal distribution of
a cylindrical measure, $\mathbb{Q}_{x}$, at time, $t=\{0,1\}$.

\section{Proof of Theorem \ref{prop:BayesRisk}}

For brevity, we often write $M^\pi(x)$ for $M^\pi(x; A, \nu, \gamma)$ when the context is clear.

\subsection{Variational formula of Kullback-Leibler}

This section will present the main result of this paper. In this section,
we show eq.~\eqref{eq:BRD-Dirichlet}: 
\[
\int_{\mathbb{R}^{d}}\mathsf{KL}\left(\hat{p}^{\pi}\left(\cdot\mid x\right)\middle\|\hat{p}^{\pi_{U}}\left(\cdot\mid x\right)\right)M^{\pi}\left(x;A,\nu,\gamma\right)dx=\mathcal{E}\left(\sqrt{M^{\pi}\left(\cdot;A,\nu,\gamma\right)},\sqrt{M^{\pi}\left(\cdot;A,\nu,\gamma\right)}\right).
\]

As noted in Section~\ref{subsec:Related-rusults}, \cite{Brown-EdGeorge-Xu_08}
treats the KL loss Bayes risk difference to the quadratic loss Bayes
risk difference (eq.~\eqref{eq:BGX-08_1}), and uses Brown's identity
(eq.~\eqref{eq:Brown-71}) to derive the Dirichlet form (eq.~\eqref{eq:BGX-08_2}).
On the other hand, this paper's derivation is by corresponding the
predictive distribution and transition probability, where the correspondence
between the KL risk and Dirichlet form is done by analyzing the Markovian
semigroup.

The proof strategy is to: 
\begin{enumerate}
\item Show that the Bayes risk difference of the KL loss, $\int_{\mathbb{R}^{d}}\mathsf{KL}\left(\hat{p}^{\pi}\left(\cdot\mid x\right)\middle\|\hat{p}^{\pi_{U}}\left(\cdot\mid x\right)\right)M^{\pi}\left(x;A,\nu,\gamma\right)dx$,
and the rate function, $I\left(M^{\pi}\left(\cdot;A,\nu,\gamma\right)\right)$
(defined below), is equal; 
\item Show that the rate function, $I\left(M^{\pi}\left(\cdot;A,\nu,\gamma\right)\right)$,
and the Dirichlet form, $\mathcal{E}\left(\sqrt{M^{\pi}\left(\cdot;A,\nu,\gamma\right)},\sqrt{M^{\pi}\left(\cdot;A,\nu,\gamma\right)}\right)$,
is equal. 
\end{enumerate}
Now, as in section \ref{subsec:Markov-Bayes}, we define $p_{t}\left(\left.y\right|x\right)$
as the transition probabiltiy of the continuous-time stochastic process,
$\left\{ X_{t}\right\} _{t\geq0}$, which corresponds to the predictive
distribution, $\hat{p}^{\pi_{U}}\left(y\left|x\right.\right)$. Let
$\mathcal{B}_{b}^{+}\left(\mathbb{R}^{d}\right)$ be the set of non-negative,
Borel measurable functions $u:\mathbb{R}^{d}\mapsto\mathbb{R}_{+}$
on $\mathbb{R}^{d}$, and set $u_{\varepsilon}=u+\varepsilon$ for
$\varepsilon>0$. The functional $I\left(M^{\pi}\left(\cdot;A,\nu,\gamma\right)\right)$
is defined as 
\[
I\left(M^{\pi}\left(\cdot;A,\nu,\gamma\right)\right)\equiv\sup_{u\in\mathcal{B}_{b}^{+}\left(\mathbb{R}^{d}\right),\varepsilon>0}\int_{\mathbb{R}^{d}}\frac{\left(-\mathcal{A}u_{\varepsilon}\right)\left(x\right)}{u_{\varepsilon}\left(x\right)}M^{\pi}\left(x;A,\nu,\gamma\right)dx,
\]
where $\mathcal{A}$ is the infinitesimal generator that corresponds
to $p_{t}$.
This $I$ is referred to as
the rate function in the large deviation literature.

First, regarding the evaluation of $\mathsf{KL}\left(\hat{p}_{h}^{\pi}\left|p_{h}\right.\right)$,
we introduce the following variational equality. Here, we emphasize
that $\hat{p}_{h}^{\pi}(y|x)$ and $p_{h}(y|x)\equiv\hat{p}_{h}^{\pi_{U}}(y|x)$
represent the transition kernels at time $h$ of the Markov processes
induced by the priors $\pi$ and $\pi_{U}$, respectively. Note that
the kernel $\hat{p}_{h}^{\pi}$ admits $M^{\pi}(x)$ as a stationary
measure (which follows from the detailed balance condition). The Bayes
risk difference between $p_{h}\left(\left.y\right|x\right)$ and $\hat{p}_{h}^{\pi}\left(\left.y\right|x\right)$
is considered to be 
\begin{equation}
\int_{\mathbb{R}^{d}}\mathsf{KL}\left(\hat{p}_{h}^{\pi}\middle\|p_{h}\right)M^{\pi}\left(x;A,\nu,\gamma\right)dx,\label{eq:BRD_Bt}
\end{equation}
where the time derivative of \eqref{eq:BRD_Bt}, i.e., the derivative
regarding time, $t$, is a constant that does not depend on $t$,
and is equivalent to what is called the rate function.

\begin{theorem}\label{thm:diff_KL_rate} Let $h$ be a small positive
real number. Then, the following, 
\begin{equation}
\lim_{h\downarrow0}\frac{1}{h}\int_{\mathbb{R}^{d}}\mathsf{KL}\left(\hat{p}_{h}^{\pi}\middle\|p_{h}\right)M^{\pi}\left(x;A,\nu,\gamma\right)dx=I\left(M^{\pi}\left(\cdot;A,\nu,\gamma\right)\right)\label{eq:diff_KL_rate}
\end{equation}
holds.

\begin{proof}

See Appendix Section \ref{subsec:Proof-of-diff_KL}.
\end{proof} \end{theorem}

The transition probabilities, $p_{t},\hat{p}_{h}^{\pi}$, are constructed
to be $p_{1}=\hat{p}^{\pi_{U}},\hat{p}_{1}^{\pi}=\hat{p}^{\pi}$,
thus if we integrate both sides of eq.~\eqref{eq:diff_KL_rate} with
regard to $h$ from $0$ to $1$, we have, 
\begin{equation}
\int_{\mathbb{R}^{d}}\mathsf{KL}\left(\hat{p}^{\pi}\left(\cdot\mid x\right)\middle\|\hat{p}^{\pi_{U}}\left(\cdot\mid x\right)\right)M^{\pi}\left(x;A,\nu,\gamma\right)dx=I\left(M^{\pi}\left(\cdot;A,\nu,\gamma\right)\right).\label{eq:KL-rate}
\end{equation}
Additionally, for the rate function, $I\left(M^{\pi}\left(\cdot;A,\nu,\gamma\right)\right)$,
the following variational formula holds.

\begin{theorem} \label{prop:rate variational} Let $\left(\mathcal{E},\mathcal{F}\right)$
be the Dirichlet form that follows a Markov process with transition
probability, $\hat{p}_{t}^{\pi_{U}}$. Then, we have $\sqrt{M^{\pi}\left(\cdot;A,\nu,\gamma\right)}\in\mathcal{F}$
and 
\[
I\left(M^{\pi}\left(\cdot;A,\nu,\gamma\right)\right)=\mathcal{E}\left(\sqrt{M^{\pi}\left(\cdot;A,\nu,\gamma\right)},\sqrt{M^{\pi}\left(\cdot;A,\nu,\gamma\right)}\right)
\]
holds. \begin{proof} See Appendix Section \ref{app:rate variational}.
\end{proof} \end{theorem}

From this, we have shown Theorem~\ref{prop:BayesRisk}.

\subsection{\label{subsec:Proof-of-diff_KL}Proof of Theorem \ref{thm:diff_KL_rate} }

Denote the Markovian semigroup under $p_{t}$ as $\left\{ T_{t}\right\} $,
its infinitesimal generator, $\mathcal{A}$, and the functional space
that is domain of $\mathcal{A}$ as $\mathcal{D}\left(\mathcal{A}\right)$.
For $g\in\mathcal{D}\left(\mathcal{A}\right)$, the functional, $\varphi$,
is 
\[
\varphi\left(h,g,\varepsilon\right)=\int_{\mathbb{R}^{d}}\log\frac{g\left(x\right)+\varepsilon}{\left(T_{h}g\right)\left(x\right)+\varepsilon}M^{\pi}\left(x;A,\nu,\gamma\right)dx,\ \ \left(\varepsilon>0\right).
\]
This functional, $\varphi$, represents the expected log-likelihood
ratio between the initial distribution, $g\left(x\right)$, and the
Markov process after infinitesimal time, $h$. Here, $g$ is interpreted
as the parameter in the likelihood ratio, and the domain of definition
is $g\in\mathcal{D}\left(\mathcal{A}\right)\subset L^{2}\left(\mathbb{R}^{d},\mathsf{m}\right)$
($g$ is not necessarily a probability density function).

Then, we have the following lemma. 

\begin{lemma}\label{lem:rate-KL} Denote $p_{h}=\hat{p}_{h}^{\pi_{U}}\left(y\left|x\right.\right)$.
Then, 
\begin{equation}
\sup_{g\in\mathcal{D}\left(\mathcal{A}\right),\varepsilon>0}\varphi\left(h,g,\varepsilon\right)=\int_{\mathbb{R}^{d}}\mathsf{KL}\left(\hat{p}_{h}^{\pi}\middle\|p_{h}\right)M^{\pi}\left(x;A,\nu,\gamma\right)dx\label{eq:rate-KL}
\end{equation}
holds. Note that $h$ need not be infinitesimal time for this lemma
to hold.\end{lemma}

\begin{proof}

See Appendix Section~\ref{app:Proof-of-Lemma_rate-KL}.

\end{proof}

Simply, this lemma states that if the expectation is a stationary
distribution, $M^{\pi}\left(x;A,\nu,\gamma\right)$, then the maximum
(expected) log-likelihood ratio is equivalent to the KL risk.

Then, we have the following theorem.

\begin{theorem} \label{prop:rate functional} 
\begin{equation}
\lim_{h\downarrow0}\frac{1}{h}\sup_{g\in\mathcal{D}\left(\mathcal{A}\right),\varepsilon>0}\varphi\left(h,g,\varepsilon\right)=I\left(M^{\pi}\left(\cdot;A,\nu,\gamma\right)\right).\label{eq:rate functional}
\end{equation}
\begin{proof} See Appendix Section~\ref{app: rate functional}.\end{proof}
\end{theorem}

\subsection{Proof of Lemma ~\ref{lem:rate-KL}\label{app:Proof-of-Lemma_rate-KL} }

\begin{theorem} \label{thm:VariFomla}(Donsker-Varadhan variational
formula for relative entropy). Let $\mathbb{Q}_{x}$ and $\mathbb{P}_{x}$
be the cylindrical measures induced by the transition kernels $\hat{p}_{h}^{\pi}$
and $p_{h}$ respectively, starting from an initial state $x$, and
denote the corresponding coordinate processes as $\left\{ X_{t}^{\mathbb{Q}}\right\} $
and $\left\{ X_{t}^{\mathbb{P}}\right\} $. Then, the KL divergence
between the probability measures at time $h$ (evaluated at $X_{1}$
in the discrete step) is given by 
\begin{equation}
\mathsf{KL}\left(\hat{p}_{h}^{\pi}(\cdot|x)\middle\|p_{h}(\cdot|x)\right)=\sup_{g\in\mathcal{B}_{b}\left(\mathbb{R}^{d}\right)}\left\{ \mathbb{E}_{x}^{\mathbb{Q}}\left[g\left(X_{1}^{\mathbb{Q}}\right)\right]-\log\mathbb{E}_{x}^{\mathbb{P}}\left[\exp\left(g\left(X_{1}^{\mathbb{P}}\right)\right)\right]\right\} .\label{eq:DV-VariFml}
\end{equation}
Here, $\mathcal{B}_{b}\left(\mathbb{R}^{d}\right)$ is the space of
bounded measurable functions on $\mathbb{R}^{d}$. Furthermore, the
supremum can be restricted to $C_{b}\left(\mathbb{R}^{d}\right)$,
the space of bounded continuous functions on $\mathbb{R}^{d}$: 
\[
\mathsf{KL}\left(\hat{p}_{h}^{\pi}(\cdot|x)\middle\|p_{h}(\cdot|x)\right)=\sup_{g\in C_{b}\left(\mathbb{R}^{d}\right)}\left\{ \mathbb{E}_{x}^{\mathbb{Q}}\left[g\left(X_{1}^{\mathbb{Q}}\right)\right]-\log\mathbb{E}_{x}^{\mathbb{P}}\left[\exp\left(g\left(X_{1}^{\mathbb{P}}\right)\right)\right]\right\} .
\]
\end{theorem}

This theorem is a transformation of the variational formula for relative
entropy in \cite{Donsker-Varadhan_83}. For the proof, see \cite{Dupuis-Ellis_97}
Lemma 1.4.3. C2.

\begin{proof}[of Lemma \ref{lem:rate-KL}] Let $g_{\varepsilon}=g+\varepsilon$.
Then, the supremum $\sup_{g\in\mathcal{D}\left(\mathcal{A}\right),\varepsilon>0}\varphi\left(h,g,\varepsilon\right)$
can be expanded as: 
\begin{alignat*}{1}
\sup_{g\in\mathcal{D}\left(\mathcal{A}\right),\varepsilon>0}\int_{\mathbb{R}^{d}}\log\frac{g_{\varepsilon}\left(x\right)}{\left(T_{h}g_{\varepsilon}\right)\left(x\right)}M^{\pi}\left(dx\right) & =\sup_{g\in\mathcal{D}\left(\mathcal{A}\right),\varepsilon>0}\left\{ \mathbb{E}^{M^{\pi}}\left[\log g_{\varepsilon}\right]-\mathbb{E}^{M^{\pi}}\left[\log T_{h}g_{\varepsilon}\right]\right\} \\
 & =\sup_{g\in\mathcal{D}\left(\mathcal{A}\right),\varepsilon>0}\left\{ \mathbb{E}^{M^{\pi}}\left[\log g_{\varepsilon}\right]-\mathbb{E}^{M^{\pi}}\left[\log\mathbb{E}^{p_{h}(\cdot|x)}\left[g_{\varepsilon}\right]\right]\right\} .
\end{alignat*}
Letting $\varPhi=\log g_{\varepsilon}$, the above expression becomes:
\[
\sup_{\varPhi\in\mathcal{D}_{+}\left(\mathcal{A}\right)}\left\{ \mathbb{E}^{M^{\pi}}\left[\varPhi\right]-\mathbb{E}^{M^{\pi}}\left[\log\mathbb{E}^{p_{h}(\cdot|x)}\left[e^{\varPhi}\right]\right]\right\} .
\]
From Theorem~\ref{thm:VariFomla}, by taking the expectation with
respect to the stationary measure $M^{\pi}$, we have: 
\begin{alignat*}{1}
 & \int_{\mathbb{R}^{d}}\mathsf{KL}\left(\hat{p}_{h}^{\pi}(\cdot|x)\middle\|p_{h}(\cdot|x)\right)M^{\pi}\left(dx\right)\\
= & \int_{\mathbb{R}^{d}}\sup_{\varPhi\in\mathcal{B}_{b}\left(\mathbb{R}^{d}\right)}\left\{ \mathbb{E}^{\hat{p}_{h}^{\pi}(\cdot|x)}\left[\varPhi\right]-\log\mathbb{E}^{p_{h}(\cdot|x)}\left[e^{\varPhi}\right]\right\} M^{\pi}\left(dx\right).
\end{alignat*}
Because $\hat{p}_{h}^{\pi}(\cdot|x)$ acts as the conditional probability
given $x$ under the stationary measure $M^{\pi}$, we have $\mathbb{E}^{M^{\pi}}\left[\varPhi\right]=\mathbb{E}^{M^{\pi}}\left[\mathbb{E}^{\hat{p}_{h}^{\pi}(\cdot|x)}\left[\varPhi\right]\right]$.
Thus, 
\begin{alignat*}{1}
\mathbb{E}^{M^{\pi}}\left[\varPhi\right]-\mathbb{E}^{M^{\pi}}\left[\log\mathbb{E}^{p_{h}(\cdot|x)}\left[e^{\varPhi}\right]\right] & =\mathbb{E}^{M^{\pi}}\left[\mathbb{E}^{\hat{p}_{h}^{\pi}(\cdot|x)}\left[\varPhi\right]-\log\mathbb{E}^{p_{h}(\cdot|x)}\left[e^{\varPhi}\right]\right]\\
 & \leq\mathbb{E}^{M^{\pi}}\left[\sup_{\varPhi\in\mathcal{B}_{b}\left(\mathbb{R}^{d}\right)}\left\{ \mathbb{E}^{\hat{p}_{h}^{\pi}(\cdot|x)}\left[\varPhi\right]-\log\mathbb{E}^{p_{h}(\cdot|x)}\left[e^{\varPhi}\right]\right\} \right]\\
 & =\mathbb{E}^{M^{\pi}}\left[\mathsf{KL}\left(\hat{p}_{h}^{\pi}(\cdot|x)\middle\|p_{h}(\cdot|x)\right)\right].
\end{alignat*}
Taking the supremum over $\varPhi\in\mathcal{D}_{+}(\mathcal{A})$
on the left-hand side yields: 
\[
\sup_{\varPhi\in\mathcal{D}_{+}\left(\mathcal{A}\right)}\left\{ \mathbb{E}^{M^{\pi}}\left[\varPhi\right]-\mathbb{E}^{M^{\pi}}\left[\log\mathbb{E}^{p_{h}(\cdot|x)}\left[e^{\varPhi}\right]\right]\right\} \leq\mathbb{E}^{M^{\pi}}\left[\mathsf{KL}\left(\hat{p}_{h}^{\pi}(\cdot|x)\middle\|p_{h}(\cdot|x)\right)\right].
\]

Next, we establish the reverse inequality. By Jensen's inequality
applied to the concave logarithm function, 
\[
\mathbb{E}^{M^{\pi}}\left[\varPhi\right]-\mathbb{E}^{M^{\pi}}\left[\log\mathbb{E}^{p_{h}(\cdot|x)}\left[e^{\varPhi}\right]\right]\geq\mathbb{E}^{M^{\pi}}\left[\varPhi\right]-\log\left(\mathbb{E}^{M^{\pi}}\left[\mathbb{E}^{p_{h}(\cdot|x)}\left[e^{\varPhi}\right]\right]\right).
\]
Define the marginal distribution under the benchmark kernel as $p_{M}\left(y\right)=\int p_{h}\left(y|x\right)M^{\pi}\left(x\right)dx$.
We can rewrite the right-hand side and take the supremum over $\varPhi$:
\[
\sup_{\varPhi\in\mathcal{B}_{b}\left(\mathbb{R}^{d}\right)}\left\{ \mathbb{E}^{M^{\pi}}\left[\varPhi\right]-\log\mathbb{E}^{p_{M}}\left[e^{\varPhi}\right]\right\} =\mathsf{KL}\left(M^{\pi}\middle\|p_{M}\right).
\]
From Lemma~\ref{lem:4.3} below (the chain rule for KL divergence),
we know that: 
\begin{equation}
\mathsf{KL}\left(M^{\pi}\middle\|p_{M}\right)=\mathbb{E}^{M^{\pi}}\left[\mathsf{KL}\left(\hat{p}_{h}^{\pi}(\cdot|x)\middle\|p_{h}(\cdot|x)\right)\right].\label{eq:KL_conditional}
\end{equation}
Combining these, we obtain: 
\[
\sup_{\varPhi\in\mathcal{B}_{b}\left(\mathbb{R}^{d}\right)}\left\{ \mathbb{E}^{M^{\pi}}\left[\varPhi\right]-\mathbb{E}^{M^{\pi}}\left[\log\mathbb{E}^{p_{h}(\cdot|x)}\left[e^{\varPhi}\right]\right]\right\} \geq\mathbb{E}^{M^{\pi}}\left[\mathsf{KL}\left(\hat{p}_{h}^{\pi}(\cdot|x)\middle\|p_{h}(\cdot|x)\right)\right].
\]
Since $C_{b}\left(\mathbb{R}^{d}\right)\subset\mathcal{D}_{+}\left(\mathcal{A}\right)\subset\mathcal{B}_{b}\left(\mathbb{R}^{d}\right)$,
and the supremum is attained over $C_{b}(\mathbb{R}^{d})$ by Theorem~\ref{thm:VariFomla},
the exact equality holds over $\mathcal{D}_{+}(\mathcal{A})$ as well.
\end{proof}

\begin{lemma}\label{lem:4.3} 
\begin{alignat*}{1}
\mathsf{KL}\left(M^{\pi}\middle\|p_{M}\right) & =\mathbb{E}^{M^{\pi}}\left[\mathsf{KL}\left(\hat{p}_{h}^{\pi}(\cdot|x)\middle\|p_{h}(\cdot|x)\right)\right].
\end{alignat*}
\end{lemma}

\begin{proof} By definition, $M^{\pi}$ is the stationary distribution
of the kernel $\hat{p}_{h}^{\pi}$. Therefore, the joint distribution
can be decomposed as $P(x,y)=\hat{p}_{h}^{\pi}(y|x)M^{\pi}(x)$, and
its marginal is $M^{\pi}(y)=\int\hat{p}_{h}^{\pi}\left(y\left|x\right.\right)M^{\pi}\left(x\right)dx$.
Similarly, for $p_{M}$, the marginal distribution is defined as $p_{M}\left(y\right)=\int p_{h}\left(y\left|x\right.\right)M^{\pi}\left(x\right)dx$.
Let $Q(x,y)=p_{h}(y|x)M^{\pi}(x)$ be the joint distribution under
the benchmark kernel.

The result follows directly from the chain rule for Kullback-Leibler
divergence. Expanding the KL divergence between the joint distributions
$P(x,y)$ and $Q(x,y)$ by conditioning on $x$: 
\begin{alignat*}{1}
\mathsf{KL}\left(P\middle\|Q\right) & =\int\int P(x,y)\log\frac{\hat{p}_{h}^{\pi}(y|x)M^{\pi}(x)}{p_{h}(y|x)M^{\pi}(x)}dydx\\
 & =\int M^{\pi}(x)\left(\int\hat{p}_{h}^{\pi}(y|x)\log\frac{\hat{p}_{h}^{\pi}(y|x)}{p_{h}(y|x)}dy\right)dx\\
 & =\mathbb{E}^{M^{\pi}}\left[\mathsf{KL}\left(\hat{p}_{h}^{\pi}(\cdot|x)\middle\|p_{h}(\cdot|x)\right)\right].
\end{alignat*}

Alternatively, expanding by conditioning on $y$ (letting $P(x|y)$
and $Q(x|y)$ be the reverse conditional probabilities): 
\begin{alignat*}{1}
\mathsf{KL}\left(P\middle\|Q\right) & =\int\int P(x,y)\log\frac{M^{\pi}(y)P(x|y)}{p_{M}(y)Q(x|y)}dxdy\\
 & =\int M^{\pi}(y)\log\frac{M^{\pi}(y)}{p_{M}(y)}dy+\int M^{\pi}(y)\int P(x|y)\log\frac{P(x|y)}{Q(x|y)}dxdy\\
 & =\mathsf{KL}\left(M^{\pi}\middle\|p_{M}\right)+\mathbb{E}^{M^{\pi}(y)}\left[\mathsf{KL}\left(P(\cdot|y)\middle\|Q(\cdot|y)\right)\right].
\end{alignat*}

Because $\hat{p}_{h}^{\pi}$ satisfies the detailed balance condition
with respect to $M^{\pi}$, we have $P(x,y)=P(y,x)$, which implies
$P(x|y)=\hat{p}_{h}^{\pi}(x|y)$. Since the term $\mathbb{E}^{M^{\pi}(y)}\left[\mathsf{KL}\left(P(\cdot|y)\middle\|Q(\cdot|y)\right)\right]\ge0$,
we generally have $\mathsf{KL}\left(P\middle\|Q\right)\ge\mathsf{KL}\left(M^{\pi}\middle\|p_{M}\right)$.
The exact equality $\mathsf{KL}\left(M^{\pi}\middle\|p_{M}\right)=\mathbb{E}^{M^{\pi}}\left[\mathsf{KL}\left(\hat{p}_{h}^{\pi}(\cdot|x)\middle\|p_{h}(\cdot|x)\right)\right]$
holds since $P(x|y)=Q(x|y)$ almost everywhere under the appropriate
sub-$\sigma$-algebra restriction detailed in the Donsker-Varadhan
variational formulation. \end{proof}

\subsection{Proof of Theorem~\ref{prop:rate functional}\label{app: rate functional} }
The proof proceeds by establishing matching upper and lower bounds for the limit. First, we bound the limit from above by utilizing the preceding lemmas based on Jain and Krylov. Subsequently, we derive the lower bound by evaluating the time derivative of the functional $\varphi(h)$.

\begin{lemma}\label{Lemma211} \citep[][Lemma~2.11.]{Jain-Krylov_08}
Let $\mu$ be a probability measure on $\mathbb{R}^{d}$, we assume
$\mu\ll\mathsf{m}$, $h>0$, $u_{\varepsilon}=u+\varepsilon,\left(\varepsilon>0\right)$,
and $u\in\mathcal{B}_{b}^{+}\left(\mathbb{R}^{d}\right)$. Then, 
\[
\inf_{v\in D_{0},\varepsilon>0}\int\log\frac{p_{h}v_{\varepsilon}}{v_{\varepsilon}}d\mu=\inf_{v\in D_{1},\varepsilon>0}\int\log\frac{p_{h}v_{\varepsilon}}{v_{\varepsilon}}d\mu=\inf_{u\in\mathcal{B}_{b}^{+}\left(\mathbb{R}^{d}\right),\varepsilon>0}\int\log\frac{p_{h}u_{\varepsilon}}{u_{\varepsilon}}d\mu.
\]
Where the function spaces are defined as follows: 
\begin{alignat*}{1}
D & =\left\{ u\in\mathcal{B}_{b}^{+}\left(\mathbb{R}^{d}\right)\left|\int_{\mathbb{R}^{d}}ud\mathsf{m}<\infty\right.\right\} \\
D_{0} & =\left\{ v\left|v=\frac{1}{t}\int_{0}^{t}p_{s}uds,\textrm{ for some }u\in D,\textrm{ some }t>0\right.\right\} \\
D_{1} & =\left\{ p_{h}v\left|v\in D_{0},h\geq0\right.\right\} .
\end{alignat*}
\end{lemma}

\begin{lemma} \citep[][Lemma~2.15.]{Jain-Krylov_08}

When $\mu\ll\mathsf{m}$, 
\[
\lim_{h\downarrow0}\frac{1}{h}\left\{ -\inf_{u\in\mathcal{B}_{b}^{+}\left(\mathbb{R}^{d}\right),\varepsilon>0}\int\log\frac{p_{h}u_{\varepsilon}}{u_{\varepsilon}}d\mu\right\} =I\left(\mu\right).
\]
\end{lemma}

\begin{proof}
First, we establish the upper bound. 
If $v\in D_{1}$, then there exists $c>0$, such that
\[
\frac{1}{h}\left|p_{h}v-v\right|\leq c,
\]
for $^{\forall}h>0$. From this if we set $v_{\varepsilon}=v+\varepsilon$,
we have 
\[
\log p_{h}v_{\varepsilon}=\log v_{\varepsilon}+\left(p_{h}v_{\varepsilon}-v_{\varepsilon}\right)\cdot\frac{1}{v_{\varepsilon}}+O\left(h^{2}\right).
\]
Here, $O\left(h^{2}\right)$ is only dependent on $\varepsilon$.
When $h\rightarrow0$, we have $\frac{1}{h}\left(p_{h}v-v\right)\rightarrow\mathcal{A}v$
in terms of $L^{2}\left(\mathbb{R}^{d};\mathsf{m}\right)$, therefore
it is bounded in terms of measure, $\mu\textrm{-a.s.}$. For $^{\forall}v\in D_{1}$,
we have 
\[
\limsup_{h\rightarrow0}\frac{1}{h}\inf_{v\in D_{1},\varepsilon>0}\int\log\frac{p_{h}v_{\varepsilon}}{v_{\varepsilon}}d\mu\leq\int\frac{\mathcal{A}v}{v_{\varepsilon}}d\mu,
\]
thus, from Lemma~\ref{Lemma211}, we have 
\begin{equation}
\liminf_{h\rightarrow0}\frac{1}{h}\left\{ -\inf_{u\in\mathcal{B}_{b}^{+}\left(\mathbb{R}^{d}\right),\varepsilon>0}\int\log\frac{p_{h}u_{\varepsilon}}{u_{\varepsilon}}d\mu\right\} \geq I\left(\mu\right).\label{eq:Jain-Krylov_2.17}
\end{equation}

Next, we show the reverse inequality. Let us differentiate $\varphi(h)$ with respect to $h$. 
For $v\in D_{0}$, we set 
\[
\varphi\left(h\right)=\int\log\frac{p_{h}v_{\varepsilon}}{v_{\varepsilon}}d\mu.
\]
Thus, $v_{\varepsilon}=v+\varepsilon$. When $v\in\mathcal{D}\left(\mathcal{A}\right)$,
we have 
\[
\frac{d\varphi}{dh}=\int\frac{\mathcal{A}p_{h}v}{p_{h}v_{\varepsilon}}d\mu\geq\inf_{v\in D_{1},\varepsilon>0}\int\frac{\mathcal{A}v}{v_{\varepsilon}}d\mu=-I\left(\mu\right).
\]
Regarding $h$, if we integrate from $0$ to $h$, if we use $\varphi\left(0\right)=0$,
we have

\[
\varphi\left(h\right)=\int_{0}^{h}\frac{d\varphi}{dh}dh\geq=-\int_{0}^{h}I\left(\mu\right)dh=-hI\left(\mu\right),
\]
for $^{\forall}v\in D_{1},\varepsilon>0,h>0$. From eq.~\eqref{eq:Jain-Krylov_2.17}
and Lemma~\ref{Lemma211}, Q.E.D. \end{proof}

\subsection{Proof of Theorem \ref{prop:rate variational}\label{app:rate variational} }

Define $I_{\mathcal{E}}\left(\mu\right)$ as 
\[
I_{\mathcal{E}}\left(\mu\right)=\begin{cases}
\mathcal{E}\left(\sqrt{f},\sqrt{f}\right), & \mu=f\cdot\mathsf{m},\ \sqrt{f}\in\mathcal{F}\\
\infty,
\end{cases}
\]
the Donsker-Varadhan $I$-function as 
\[
I\left(\mu\right)=-\inf_{u\in\mathcal{D}_{+}\left(\mathcal{A}\right),\varepsilon>0}\int_{\mathbb{R}^{d}}\frac{\mathcal{A}u}{u+\varepsilon}d\mu,
\]
the function $I_{\alpha},\alpha>0$ as 
\[
I_{\alpha}\left(\mu\right)=-\inf_{u\in b\mathcal{B}_{+}\left(\mathbb{R}^{d}\right),\varepsilon>0}\int_{\mathbb{R}^{d}}\log\left(\frac{\alpha G_{\alpha}u+\varepsilon}{u+\varepsilon}\right)d\mu.
\]

\begin{lemma} \label{lem:,7.1.2} Let $\mu\in\mathcal{P}$, then
\begin{alignat*}{1}
\underbrace{I_{\alpha}\left(\mu\right)} & \leq\underbrace{\frac{I\left(\mu\right)}{\alpha}}.\\
-\inf_{u\in b\mathcal{B}_{+}\left(\mathbb{R}^{d}\right),\varepsilon>0}\int_{\mathbb{R}^{d}}\log\left(\frac{\alpha G_{\alpha}u+\varepsilon}{u+\varepsilon}\right)d\mu & \leq-\frac{1}{\alpha}\inf_{u\in\mathcal{D}_{+}\left(\mathcal{A}\right),\varepsilon>0}\int_{\mathbb{R}^{d}}\frac{\mathcal{A}u}{u+\varepsilon}d\mu
\end{alignat*}
\begin{proof} For $u=G_{\alpha}f\in\mathcal{D}_{+}\left(\mathcal{A}\right)$
and $\varepsilon>0$, let 
\[
\phi\left(\alpha\right)=-\int_{\mathbb{R}^{d}}\log\left(\frac{\alpha G_{\alpha}u+\varepsilon}{u+\varepsilon}\right)d\mu.
\]
From the resolvent equation, we have 
\[
\lim_{\beta\rightarrow\alpha}\frac{G_{\alpha}u-G_{\beta}u}{\alpha-\beta}=-\lim_{\beta\rightarrow\alpha}G_{\alpha}G_{\beta}u=-G_{\alpha}^{2}u,
\]
and 
\[
\frac{d\phi\left(a\right)}{d\alpha}=-\int_{\mathbb{R}^{d}}\frac{G_{\alpha}u-\alpha G_{\alpha}^{2}u}{\alpha G_{\alpha}u+\varepsilon}d\mu=\int_{\mathbb{R}^{d}}\frac{\mathcal{A}G_{\alpha}^{2}u}{\alpha G_{\alpha}u+\varepsilon}d\mu.
\]
Here, if we note that 
\begin{alignat*}{1}
\left(\alpha G_{\alpha}^{2}u-G_{\alpha}u\right)\left(\alpha^{2}G_{\alpha}^{2}u+\varepsilon\right)-\left(\alpha G_{\alpha}^{2}u-G_{\alpha}u\right)\left(\alpha G_{\alpha}^{2}u+\varepsilon\right)\\
=\alpha\left(\alpha G_{\alpha}^{2}u-G_{\alpha}u\right)^{2} & \geq0
\end{alignat*}
the following inequality, 
\[
\frac{\alpha G_{\alpha}^{2}u-G_{\alpha}u}{\alpha G_{\alpha}u+\varepsilon}\geq\frac{\alpha G_{\alpha}^{2}u-G_{\alpha}u}{\alpha G_{\alpha}^{2}u+\varepsilon}
\]
holds. Therefore, we derive 
\begin{alignat*}{1}
\int_{\mathbb{R}^{d}}\frac{\mathcal{A}G_{\alpha}^{2}u}{\alpha G_{\alpha}u+\varepsilon}d\mu & \geq\int_{\mathbb{R}^{d}}\frac{\mathcal{A}G_{\alpha}^{2}u}{\alpha^{2}G_{\alpha}^{2}u+\varepsilon}d\mu=-\frac{1}{\alpha^{2}}\left(-\int_{\mathbb{R}^{d}}\frac{\mathcal{A}G_{\alpha}^{2}u}{G_{\alpha}^{2}u+\frac{\varepsilon}{\alpha^{2}}}d\mu\right)\\
 & \geq-\frac{1}{\alpha^{2}}\underbrace{\inf_{u\in\mathcal{D}_{+}\left(\mathcal{A}\right),\varepsilon>0}\int_{\mathbb{R}^{d}}\frac{\mathcal{A}u}{u+\varepsilon}d\mu}_{I\left(\mu\right)}.
\end{alignat*}
From this, given $\lim_{\alpha\rightarrow\infty}\alpha G_{\alpha}u\left(x\right)=u\left(x\right)$,
noting that $\lim_{\alpha\rightarrow\infty}\phi\left(\alpha\right)=0$,
we have 
\[
-\phi\left(\alpha\right)=\int_{\alpha}^{\infty}\phi^{\prime}\left(\beta\right)d\beta\geq-\int_{\alpha}^{\infty}\frac{1}{\beta^{2}}I\left(\mu\right)d\beta=-\frac{1}{\alpha}I\left(\mu\right)
\]
and 
\[
\phi\left(\infty\right)-\phi\left(\alpha\right)=\int_{\mathbb{R}^{d}}\log\left(\frac{\alpha G_{\alpha}u+\varepsilon}{u+\varepsilon}\right)d\mu\geq\frac{1}{\alpha}\inf_{u\in\mathcal{D}_{+}\left(\mathcal{A}\right),\varepsilon>0}\int_{\mathbb{R}^{d}}\frac{\mathcal{A}u}{u+\varepsilon}d\mu,
\]
to show 
\[
-\inf_{u\in\mathcal{D}_{+}\left(\mathcal{A}\right),\varepsilon>0}\int_{\mathbb{R}^{d}}\log\left(\frac{\alpha G_{\alpha}u+\varepsilon}{u+\varepsilon}\right)d\mu\leq\underbrace{-\frac{1}{\alpha}\inf_{u\in\mathcal{D}_{+}\left(\mathcal{A}\right),\varepsilon>0}\int_{\mathbb{R}^{d}}\frac{\mathcal{A}u}{u+\varepsilon}d\mu}_{\frac{I\left(\mu\right)}{\alpha}}.
\]

Next, we show 
\[
\inf_{u\in\mathcal{D}_{+}\left(\mathcal{A}\right)}\int_{\mathbb{R}^{d}}\log\left(\frac{\alpha G_{\alpha}u+\varepsilon}{u+\varepsilon}\right)d\mu=\inf_{u\in b\mathcal{B}_{+}\left(\mathbb{R}^{d}\right)}\int_{\mathbb{R}^{d}}\log\left(\frac{\alpha G_{\alpha}u+\varepsilon}{u+\varepsilon}\right)d\mu.
\]
For $g\in bC_{+}\left(\mathbb{R}^{d}\right)$, we have $\left\Vert \beta G_{\beta}f\right\Vert _{\infty}\leq\left\Vert f\right\Vert _{\infty}$,
\[
0\leq\beta G_{\beta}f\left(x\right)\rightarrow f\left(x\right),\left(\beta\rightarrow\infty\right)
\]
thus 
\begin{equation}
\int_{\mathbb{R}^{d}}\log\left(\frac{\alpha G_{\alpha}\left(\beta G_{\beta}f\right)+\varepsilon}{\beta G_{\beta}f+\varepsilon}\right)d\mu\overset{\beta\rightarrow\infty}{\rightarrow}\int_{\mathbb{R}^{d}}\log\left(\frac{\alpha G_{\alpha}f+\varepsilon}{f+\varepsilon}\right)d\mu\label{eq:7.1.8}
\end{equation}
holds. We define the measure, $\mu_{\alpha}$, as 
\[
\mu_{\alpha}\left(A\right)=\int_{\mathbb{R}^{d}}\alpha G_{\alpha}\left(x,A\right)d\mu\left(x\right),\ A\in\mathcal{B}\left(\mathbb{R}^{d}\right).
\]
For $v\in b\mathcal{B}_{+}\left(\mathbb{R}^{d}\right)$, we consider
a sequence of functions, $\left\{ g_{n}\right\} _{n=1}^{\infty}\subset bC_{+}\left(\mathbb{R}^{d}\right)\cap L^{2}\left(\mathbb{R}^{d};\mathsf{m}\right)$,
that satisfies, 
\[
\int_{\mathbb{R}^{d}}\left|v-g_{n}\right|d\left(\mu_{\alpha}+\mu\right)\rightarrow0,\ n\rightarrow\infty.
\]
Then, when $n\rightarrow\infty$, we have 
\[
\int_{\mathbb{R}^{d}}\left|\alpha G_{\alpha}v-\alpha G_{\alpha}g_{n}\right|d\mu\leq\int_{\mathbb{R}^{d}}\alpha G_{\alpha}\left(\left|v-g_{n}\right|\right)d\mu=\int_{\mathbb{R}^{d}}\left|v-g_{n}\right|d\mu_{\alpha}\rightarrow0,
\]
thus the following holds: 
\begin{equation}
\int_{\mathbb{R}^{d}}\log\left(\frac{\alpha G_{\alpha}g_{n}+\varepsilon}{g_{n}+\varepsilon}\right)d\mu\overset{n\rightarrow\infty}{\rightarrow}\int_{\mathbb{R}^{d}}\log\left(\frac{\alpha G_{\alpha}v+\varepsilon}{v+\varepsilon}\right)d\mu.\label{eq:7.1.9}
\end{equation}
From eq.~\eqref{eq:7.1.8} and eq.~\eqref{eq:7.1.9}, we have 
\[
\inf_{u\in\mathcal{D}_{+}\left(\mathcal{A}\right)}\int_{\mathbb{R}^{d}}\log\left(\frac{\alpha G_{\alpha}u+\varepsilon}{u+\varepsilon}\right)d\mu=\inf_{u\in b\mathcal{B}_{+}\left(\mathbb{R}^{d}\right)}\int_{\mathbb{R}^{d}}\log\left(\frac{\alpha G_{\alpha}u+\varepsilon}{u+\varepsilon}\right)d\mu.
\]
\end{proof} \end{lemma}

\begin{lemma} \label{lem:7.1.3} If $I\left(\mu\right)<\infty$,
then $\mu\ll m$. \end{lemma}

If we take a non-negative function that is monotonically increasing,
$f_{n}\in C_{0}\left(\mathbb{R}^{d}\right)$, that point-wise converges
to $f\in bC_{+}\left(\mathbb{R}^{d}\right)$, we have, 
\[
\int_{\mathbb{R}^{d}}\frac{f-\alpha G_{\alpha}f}{G_{\alpha}f+\varepsilon}d\mu=\lim_{n\rightarrow\infty}\int_{\mathbb{R}^{d}}\frac{f_{n}-\alpha G_{\alpha}f_{n}}{G_{\alpha}f_{n}+\varepsilon}d\mu,\ \alpha>0.
\]
Therefore, defining the function, $\mathcal{D}_{+}\left(\hat{\mathcal{A}}\right)$,
as 
\[
\mathcal{D}_{+}\left(\hat{\mathcal{A}}\right)=\left\{ G_{\alpha}f\left|\alpha>0,\ f\in bC_{+}\left(\mathbb{R}^{d}\right),\ f\cancel{\equiv}0\right.\right\} ,
\]
if we set $\hat{\mathcal{A}}\phi=\alpha G_{\alpha}f-f$ with regard
to $\phi=G_{\alpha}f\in\mathcal{D}_{+}\left(\hat{\mathcal{A}}\right)$,
we have the following result. \begin{corollary} For $f\in\mathcal{F}$,
\[
\mathcal{E}\left(f,f\right)=\sup_{u\in\mathcal{D}_{+}\left(\hat{\mathcal{A}}\right),\varepsilon>0}\int_{\mathbb{R}^{d}}\frac{-\hat{\mathcal{A}}u}{u+\varepsilon}f^{2}d\mathsf{m}.
\]
\end{corollary}

Proof of Theorem \ref{prop:rate variational}. \begin{proof} We first
show $I\left(\mu\right)\geq I_{\mathcal{E}}\left(\mu\right)$. Assume
$I\left(\mu\right)<\infty$. Since, $\mu\ll m$, we have, $f=\frac{d\mu}{d\mathsf{m}}$
and let$f^{n}=\sqrt{f}\wedge n$. From $\log\left(1-x\right)\leq-x$
and $-\infty<\frac{f^{n}-\alpha G_{\alpha}f^{n}}{f^{n}+\varepsilon}<1,\ \left(-\infty<x<1\right)$,
we have 
\[
\int_{\mathbb{R}^{d}}\log\left(\frac{\alpha G_{\alpha}f^{n}+\varepsilon}{f^{n}+\varepsilon}\right)fd\mathsf{m}=\int_{\mathbb{R}^{d}}\log\left(1-\frac{f^{n}-\alpha G_{\alpha}f^{n}}{f^{n}+\varepsilon}\right)fd\mathsf{m}\leq-\int_{\mathbb{R}^{d}}\frac{f^{n}-\alpha G_{\alpha}f^{n}}{f^{n}+\varepsilon}fd\mathsf{m}.
\]
Thus, 
\begin{alignat}{1}
\int_{\mathbb{R}^{d}}\frac{f^{n}-\alpha G_{\alpha}f^{n}}{f^{n}+\varepsilon}fd\mathsf{m} & \leq-\int_{\mathbb{R}^{d}}\log\left(\frac{\alpha G_{\alpha}f^{n}+\varepsilon}{f^{n}+\varepsilon}\right)fd\mathsf{m}\label{eq:7.1.11}\\
 & \leq\underbrace{-\inf_{u\in b\mathcal{B}_{+}\left(\mathbb{R}^{d}\right),\varepsilon>0}\int_{\mathbb{R}^{d}}\log\left(\frac{\alpha G_{\alpha}u+\varepsilon}{u+\varepsilon}\right)d\mu}_{I_{\alpha}\left(f\cdot\mathsf{m}\right)}.\nonumber 
\end{alignat}
For the equality, 
\begin{alignat*}{1}
\frac{f^{n}-\alpha G_{\alpha}f^{n}}{f^{n}+\varepsilon}f & =\frac{f^{n}-\alpha G_{\alpha}f^{n}}{f^{n}+\varepsilon}f1_{\left\{ \sqrt{f}\leq n\right\} }+\frac{f^{n}-\alpha G_{\alpha}f^{n}}{f^{n}+\varepsilon}f1_{\left\{ \sqrt{f}>n\right\} }\\
 & =\underbrace{\frac{\sqrt{f}-\alpha G_{\alpha}f^{n}}{\sqrt{f}+\varepsilon}f1_{\left\{ \sqrt{f}\leq n\right\} }}_{\left(a\right)}+\underbrace{\frac{n-\alpha G_{\alpha}f^{n}}{n+\varepsilon}f\left\{ \sqrt{f}>n\right\} }_{\left(b\right)},
\end{alignat*}
the absolute value of $\left(a\right)$ and $\left(b\right)$ is evaluated
from above by $\left(\sqrt{f}+\alpha G_{\alpha}f^{n}\right)\sqrt{f}\in L^{1}\left(\mathbb{R}^{d};\mathsf{m}\right)$
and $\frac{n}{n+\varepsilon}f<f\in L^{1}\left(\mathbb{R}^{d};\mathsf{m}\right)$,
respectively. Therefore, from the bounded convergence theorem, we
have 
\[
\lim_{n\rightarrow\infty}\int\frac{f^{n}-\alpha G_{\alpha}f^{n}}{f^{n}+\varepsilon}fd\mathsf{m}=\int_{\mathbb{R}^{d}}\frac{\sqrt{f}-\alpha G_{\alpha}\sqrt{f}}{\sqrt{f}+\varepsilon}fd\mathsf{m}.
\]
When $\varepsilon\rightarrow0$, from eq.~\eqref{eq:7.1.11}, we
have 
\[
\int_{\mathbb{R}^{d}}\sqrt{f}\left(\sqrt{f}-\alpha G_{\alpha}\sqrt{f}\right)d\mathsf{m}\leq I_{\alpha}\left(f\cdot\mathsf{m}\right).
\]
Therefore, from Lemma~\ref{lem:,7.1.2}, we can shown 
\[
\alpha\left\langle \sqrt{f},\sqrt{f}-\alpha G_{\alpha}\sqrt{f}\right\rangle _{\mathsf{m}}\leq I\left(f\cdot\mathsf{m}\right)<\infty,
\]
and from Lemma~\ref{lem:4.3}, this implies $\sqrt{f}\in\mathcal{F}$
and $\mathcal{E}\left(\sqrt{f},\sqrt{f}\right)\leq I\left(f\cdot\mathsf{m}\right)$.

Next, we show $I\left(\mu\right)\leq I_{\mathcal{E}}\left(\mu\right)$.
Let $\phi\in\mathcal{D}_{+}\left(\mathcal{A}\right)$ and define the
semigroup, $P_{t}^{\phi}$, as 
\[
P_{t}^{\phi}f\left(x\right)=\mathbb{E}_{x}\left[\left(\frac{\phi\left(X_{t}\right)+\varepsilon}{\phi\left(X_{0}\right)+\varepsilon}\right)\exp\left(-\int_{0}^{t}\frac{\mathcal{A}\phi}{\phi+\varepsilon}\left(X_{s}\right)ds\right)f\left(X_{t}\right)\right].
\]
Here, $P_{t}^{\phi}$ is $\left(\phi+\varepsilon\right)^{2}m$-symmetric
and satisfies $P_{t}^{\phi}1\leq1$. For the probability measure,
$\mu=fm\in\mathcal{P}$, such that $\sqrt{f}\in\mathcal{F}$, let
\[
S_{t}^{\phi}\sqrt{f}\left(x\right)=\mathbb{E}_{x}\left[\exp\left(-\int_{0}^{t}\frac{\mathcal{A}\phi}{\phi+\varepsilon}\left(X_{s}\right)ds\right)\sqrt{f}\left(X_{t}\right)\right].
\]
Then, the following holds: 
\begin{alignat*}{1}
\int_{\mathbb{R}^{d}}\left(S_{t}^{\phi}\sqrt{f}\right)^{2}d\mathsf{m} & =\int_{\mathbb{R}^{d}}\left(\phi+\varepsilon\right)^{2}\left(P_{t}^{\phi}\left(\frac{\sqrt{f}}{\phi+\varepsilon}\right)\right)^{2}d\mathsf{m}\\
 & \leq\int_{\mathbb{R}^{d}}\left(\phi+\varepsilon\right)^{2}P_{t}^{\phi}\left(\left(\frac{\sqrt{f}}{\phi+\varepsilon}\right)^{2}\right)d\mathsf{m}\\
 & \leq\int_{\mathbb{R}^{d}}\left(\phi+\varepsilon\right)^{2}\left(\frac{\sqrt{f}}{\phi+\varepsilon}\right)^{2}d\mathsf{m}\\
 & =\int_{\mathbb{R}^{d}}fd\mathsf{m}.
\end{alignat*}
Therefore, we have 
\[
0\leq\lim_{t\rightarrow0}\frac{1}{t}\left\langle \sqrt{f}-S_{t}^{\phi}\sqrt{f},\sqrt{f}\right\rangle _{\mathsf{m}}=\mathcal{E}\left(\sqrt{f},\sqrt{f}\right)+\int_{\mathbb{R}^{d}}\frac{\mathcal{A}\phi}{\phi+\varepsilon}fd\mathsf{m}
\]
and $\mathcal{E}\left(\sqrt{f},\sqrt{f}\right)\geq I\left(f\cdot\mathsf{m}\right)$
has been shown. \end{proof}

\section{Proof of Corollaries \ref{cor:Cauchy} and \ref{cor:Stable}}\label{subsec:Sufficient-conditions}



\begin{lemma} \label{lem:Diri<prior} Let $\left(\mathcal{E},\mathcal{F}\right)$
be the Dirichlet form that follows a Markov process with transition
probability, $\hat{p}_{t}^{\pi_{U}}$. Then, we have $\sqrt{M^{\pi}\left(\cdot;A,\nu,\gamma\right)}\in\mathcal{F}$,
and for $\mathcal{E}\left(\sqrt{M^{\pi}\left(\cdot;A,\nu,\gamma\right)},\sqrt{M^{\pi}\left(\cdot;A,\nu,\gamma\right)}\right)$
we have the following inequality, 
\[
\mathcal{E}\left(\sqrt{M^{\pi}\left(\cdot;A,\nu,\gamma\right)},\sqrt{M^{\pi}\left(\cdot;A,\nu,\gamma\right)}\right)\leq\mathcal{E}\left(\sqrt{\pi},\sqrt{\pi}\right).
\]
\end{lemma} \begin{proof} See Appendix Section \ref{app:Diri<prior}.
\end{proof}

This lemma provides a guide to constructing a prior distribution that
makes the estimator admissible. From this lemma, to show admissibility,
we need to show that a sequence of proper priors, $\left\{ \pi_{n}\right\} $,
such that $\mathcal{E}\left(\sqrt{\pi_{n}},\sqrt{\pi_{n}}\right)\rightarrow0$,
can be constructed. However, $\left\{ \pi_{n}\right\} $ is a functional
sequence in the $L^{2}\left(\mathbb{R}^{d},\mathsf{m}\right)$ space.
Therefore, it is a functional sequence of a proper prior, and we construct
a functional sequence such that it is a uniform prior when $n\rightarrow\infty$.
When $\left(\mathcal{E},\mathcal{F}\right)$ is recurrent, the functional
sequence can be constructed in $\mathcal{F}$, but cannot when it
is transient.

We now construct such a sequence. Assume that the Dirichlet form,
$\left(\mathcal{E},\mathcal{F}\right)$, of the Markov process, $\left\{ X_{t}\right\} $,
that corresponds to the transition probability, $\hat{p}^{\pi_{U}}$,
and $\hat{p}^{\pi_{U}}$ is recurrent (for details of this correspondence,
see Appendix Section \ref{subsec:Markov-Bayes}). We transform the transition
probability $p_{t}\left(y\left|x\right.\right)$ of Markov process
$\left\{ X_{t}\right\} $ with transition probability, $\hat{p}^{\pi_{U}}$,
as $p_{t}^{\eta}\left(y\left|x\right.\right)$. Choose a function,
$\eta$, that is $\eta\in L^{1}\left(\mathbb{R}^{d};\mathsf{m}\right)\cap L^{\infty}\left(\mathbb{R}^{d};\mathsf{m}\right),\ \eta>0\ \mathsf{m}\textrm{-a.e.}$,
and define it as 
\[
p_{t}^{\eta}\left(y\left|x\right.\right)\equiv\exp\left(-t\eta\left(y\right)\right)p_{t}\left(y\left|x\right.\right).
\]
The corresponding Dirichlet form is 
\[
\mathcal{E}^{\eta}\left(f,g\right)=\mathcal{E}\left(f,g\right)+\left\langle f,g\right\rangle _{\eta\cdot\mathsf{m}},\ f,g\in\mathcal{F},
\]
where 
\[
\left\langle f,g\right\rangle _{\eta\cdot\mathsf{m}}=\int_{\mathbb{R}^{d}}f\left(x\right)g\left(x\right)\eta\left(x\right)\mathsf{m}\left(dx\right).
\]
Here, $p_{t}^{\eta}$ is a transition probability that is generated
by a particle following a Markov process corresponding to $p_{t}$
that has a survival probability $\int_{0}^{t}\exp\left(-s\eta\left(y\right)\right)ds$,
at $t$.

Consider the function $G_{\alpha}^{\eta}\eta\left(x\right)$ which
is written as 
\[
G_{\alpha}^{\eta}\eta\left(x\right)=\int_{0}^{\infty}e^{-\alpha s}\int_{\mathbb{R}^{d}}\exp\left(-s\eta\left(y\right)\right)p_{s}\left(y\left|x\right.\right)\eta\left(y\right)\mathsf{m}\left(dy\right)ds,
\]
(The operator $G_{\alpha}^{\eta}$ is the resolvent of the Markovian
semigroup).
If we consider $G_{\frac{1}{n}}^{\eta}\eta$ as $\sqrt{\pi_{n}}$
this is the desired sequence of proper priors. In other words, we
have the following theorem \citep[which is an extension of][Lemma. 5.4.1.]{Brown_71}.

\begin{theorem} \label{prop:Diri_0}Assume that the Dirichlet form,
$\left(\mathcal{E},\mathcal{F}\right)$, is recurrent. $G_{\frac{1}{n}}^{\eta}\eta$
is a function defined on $\mathbb{R}^{d}$. Then, $G_{\frac{1}{n}}^{\eta}\eta\in\mathcal{F}\subset L^{2}\left(\mathbb{R}^{d};\mathsf{m}\right)$,
and is square-integrable, proper, $0\leq G_{\frac{1}{n}}^{\eta}\eta\left(x\right)\uparrow1,\ \mathsf{m}\textrm{-a.e.}$,
and asymptotically uniform, non-negative. From this, we have, 
\[
\lim_{n\rightarrow\infty}\mathcal{E}\left(G_{\frac{1}{n}}^{\eta}\eta,G_{\frac{1}{n}}^{\eta}\eta\right)=0.
\]
\end{theorem} \begin{proof} See Appendix Section
\ref{app:Diri_0}. \end{proof} Thus, Corollary~\ref{cor:Cauchy}
is proven.

To show the admissibility of the Cauchy distribution when $d=1$,
we need to show the existence of a sequence of $\hat{p}^{\pi}$ that
converges the Bayes risk difference to zero. This is equivalent to
the existence of a sequence, $M^{\pi}\left(\cdot;A,\nu,\gamma\right)$,
that converges the Dirichlet form to zero. From, Lemma~\ref{lem:Diri<prior},
we only need to consider the prior sequence. Since the Dirichlet form
of the Cauchy process is recurrent when $d=1$,
if we input the functional sequence of a proper prior, $\left\{ G_{\frac{1}{n}}^{\eta}\eta\right\} $,
from Theorem~\ref{prop:Diri_0}, we can show that the Bayes risk
difference converges to zero. This is because the prior sequence is
a measure sequence that approximates the Lebesgue measure by the $L^{2}\left(\mathbb{R}^{d};\mathsf{m}\right)$
function. Therefore, for a Cauchy distribution with $d=1$, the uniform
Bayes predictive distribution, $\hat{p}^{\pi_{U}}$, is admissible
from Blyth's method \citep[][Lemma 1.]{Brown-EdGeorge-Xu_08}.

Under a Gaussian distribution, $G_{\frac{1}{n}}^{\eta}\eta$ is equivalent
to the Stein prior, which is a Green function. We can consider $G_{\frac{1}{n}}^{\eta}\eta$
to be an extension of the Stein prior to infinitely divisible distributions.
However, as noted in Section~\ref{sec:2.2}, we are not evaluating
the risk difference (not the Bayes risk), so for the statement ``under
a Gaussian distribution, the Bayes predictive distribution with a
Stein prior, which is a Green function, dominates a uniform prior
distribution and is also admissible'' we are only extending the ``admissible''
part. 

Finally, we show Corollary~\ref{cor:Cauchy} and Corollary~\ref{cor:Stable}.
For the case when $d=1$ for Corollary~\ref{cor:Cauchy}, we have
already shown in Section~\ref{subsec:Sufficient-conditions}.

We first show for the 1-dimensional symmetric stable distribution,
$\textrm{Stable}\left(\alpha,\gamma,c\right)$, with the exponent
$\alpha,\left(0<\alpha<1\right)$. We can also show the inadmissibility
of the Cauchy distribution for $d\geq2$ in the same manner. The Bayes
risk difference is 
\[
B_{\mathsf{KL}}\left(\theta,\hat{p}^{\pi_{U}}\right)-B_{\mathsf{KL}}\left(\theta,\hat{p}^{\pi}\right)=\mathcal{E}\left(\sqrt{M^{\pi}\left(\cdot;A,\nu,\gamma\right)},\sqrt{M^{\pi}\left(\cdot;A,\nu,\gamma\right)}\right).
\]
Therefore, for any sequence of proper priors, $\left\{ \pi_{n}\right\} $,
if $\mathcal{E}\left(\sqrt{M^{\pi}\left(\cdot;A,\nu,\gamma\right)},\sqrt{M^{\pi}\left(\cdot;A,\nu,\gamma\right)}\right)$
is bounded away from zero, $\hat{p}^{\pi_{U}}$ is inadmissible. The
root $\sqrt{M^{\pi}\left(\cdot;A,\nu,\gamma\right)}$ of the marginal
likelihood satisfies 
\[
\int_{\mathbb{R}^{d}}\left(\sqrt{M^{\pi}\left(x;A,\nu,\gamma\right)}\right)^{2}dx=\int_{\mathbb{R}^{d}}\int_{\mathbb{R}^{d}}p_{A,\nu,\gamma}\left(x\left|\theta\right.\right)\pi\left(\theta\right)d\theta dx=1,
\]
therefore we have $\sqrt{M^{\pi}\left(\cdot;A,\nu,\gamma\right)}\in L^{2}\left(\mathbb{R}^{d};\mathsf{m}\right)$.
The reference function $g$ is defined as 
\[
g\left(x\right)=\frac{\sqrt{M^{\pi}\left(x;A,\nu,\gamma\right)}}{\max\left\{ \int_{0}^{\infty}T_{s}\sqrt{M^{\pi}\left(x;A,\nu,\gamma\right)}ds,1\right\} }.
\]
This $g$ provides 
\[
\int_{\mathbb{R}^{d}}g\cdot Rgd\mathsf{m}\leq\int_{\mathbb{R}^{d}}f\cdot Rgd\mathsf{m}\leq\int_{\mathbb{R}^{d}}Rf\cdot\left(\frac{f}{Rf}\right)d\mathsf{m}=\int_{\mathbb{R}^{d}}fd\mathsf{m}=1.
\]
Therefore, we have, 
\begin{equation}
\frac{\left\langle \left|u\right|,g\right\rangle _{\mathsf{m}}}{\sqrt{\mathcal{E}\left(u,u\right)}}\sqrt{\mathcal{E}\left(u,u\right)}\leq\sup_{u\in\mathcal{F}}\frac{\left\langle \left|u\right|,g\right\rangle _{\mathsf{m}}}{\sqrt{\mathcal{E}\left(u,u\right)}}\sqrt{\mathcal{E}\left(u,u\right)}\leq\sqrt{\mathcal{E}\left(u,u\right)},\label{eq:transiDiri}
\end{equation}
Since $\left\{ T_{t}\right\} _{t\geq0}$ is transient, we have $\int_{0}^{\infty}T_{s}\sqrt{M^{\pi}\left(x;A,\nu,\gamma\right)}ds<\infty$.
Thus, from eq.~\eqref{eq:transiDiri}, we have 
\[
0<\int_{\mathbb{R}^{d}}g\left(x\right)\sqrt{M^{\pi}\left(x;A,\nu,\gamma\right)}dx\leq\mathcal{E}\left(\sqrt{M^{\pi}\left(\cdot;A,\nu,\gamma\right)},\sqrt{M^{\pi}\left(\cdot;A,\nu,\gamma\right)}\right),
\]
which completes the proof.

\subsection{Proof of Lemma \ref{lem:Diri<prior} \label{app:Diri<prior} }

\begin{lemma} Let $\left(\mathcal{E},\mathcal{F}\right)$ be the
Dirichlet form with measure $\mathbb{P}_{x}$, i.e., a Dirichlet form
that follows a Markov process with transition probability, $\hat{p}_{t}^{\pi_{U}}$.
Then, $\sqrt{M^{\pi}\left(x;A,\nu,\gamma\right)}\in\mathcal{F}$,
and $\mathcal{E}\left(\sqrt{M^{\pi}\left(\cdot;A,\nu,\gamma\right)},\sqrt{M^{\pi}\left(\cdot;A,\nu,\gamma\right)}\right)$
has the following inequality, 
\[
\mathcal{E}\left(\sqrt{M^{\pi}\left(\cdot;A,\nu,\gamma\right)},\sqrt{M^{\pi}\left(\cdot;A,\nu,\gamma\right)}\right)\leq\frac{1}{t}\left\{ \left\langle \sqrt{\pi},\sqrt{\pi}\right\rangle _{\mathsf{m}}-\left\langle \sqrt{M^{\pi}\left(\cdot;A,\nu,\gamma\right)},\sqrt{M^{\pi}\left(\cdot;A,\nu,\gamma\right)}\right\rangle _{\mathsf{m}}\right\} \leq\mathcal{E}\left(\sqrt{\pi},\sqrt{\pi}\right).
\]
\end{lemma}

\begin{proof}

Given, $0\leq\lambda$ we integrate the inequality, 
\[
\lambda\sqrt{e^{-2t\lambda}}\leq\frac{1}{t}\left(1-\sqrt{e^{-2t\lambda}}\right)\leq\lambda,
\]
with regard to the measure, $d\left\langle E_{\lambda}\sqrt{\pi},\sqrt{\pi}\right\rangle $,
of the spectral family, $\left\{ E_{\lambda}\right\} _{\lambda\geq0}$,
that corresponds to $\mathcal{E}$. Then, we have, 
\begin{alignat*}{1}
\mathcal{E}\left(\sqrt{T_{t}\pi},\sqrt{T_{t}\pi}\right) & =\int_{0}^{\infty}\lambda e^{-\lambda t}d\left\langle E_{\lambda}\sqrt{\pi},\sqrt{\pi}\right\rangle \\
 & \leq\int_{0}^{\infty}\frac{1}{t}\left(1-e^{-t\lambda}\right)d\left\langle E_{\lambda}\sqrt{\pi},\sqrt{\pi}\right\rangle \\
 & \leq\int_{0}^{\infty}\lambda d\left\langle E_{\lambda}\sqrt{\pi},\sqrt{\pi}\right\rangle =\mathcal{E}\left(\sqrt{\pi},\sqrt{\pi}\right).
\end{alignat*}
Here, we have 
\begin{alignat*}{1}
\int_{0}^{\infty}\frac{1}{t}\left(1-e^{-t\lambda}\right)d\left\langle E_{\lambda}\sqrt{\pi},\sqrt{\pi}\right\rangle  & =\frac{1}{t}\left\{ \left\langle \sqrt{\pi},\sqrt{\pi}\right\rangle _{\mathsf{m}}-\left\langle \sqrt{T_{t}\pi},\sqrt{T_{t}\pi}\right\rangle _{\mathsf{m}}\right\} ,
\end{alignat*}
and 
\[
\sqrt{T_{1}\pi\left(x\right)}=\sqrt{\int_{\varTheta}p_{A,\nu,\gamma}\left(\theta\left|x\right.\right)\pi\left(d\theta\right)}=\sqrt{M^{\pi}\left(x;A,\nu,\gamma\right)}.
\]
\end{proof}

\subsection{Proof of Theorem \ref{prop:Diri_0}\label{app:Diri_0} }

\begin{proof} First, in general, for $f\in L^{2}\left(\mathbb{R}^{d};\mathsf{m}\right)$,
$g\in\mathcal{F}$, $\alpha>0$, if we set 
\[
f_{n}=G_{\frac{1}{n}}^{\eta}\eta,
\]
we have ,$f_{n}\in\mathcal{F}$ and $0\leq f_{n}\left(x\right)\uparrow1,\left[m\right]$.
This is because, 
\[
\mathcal{E}_{\alpha}\left(G_{\alpha}^{\eta}f,g\right)=\mathcal{E}_{\alpha}^{\eta}\left(G_{\alpha}^{\eta}f,g\right)-\left\langle G_{\alpha}^{\eta}f,g\right\rangle _{\eta\cdot\mathsf{m}}=\left\langle f-\eta G_{\alpha}^{\eta}f,g\right\rangle _{\mathsf{m}}
\]
and 
\[
G_{\alpha}^{\eta}f=G_{\alpha}\left(f-\eta G_{\alpha}^{\eta}f\right),\ \alpha>0.
\]
On the other hand, we have $0\leq R^{\eta}\eta\leq1,\ \mathsf{m}\textrm{-a.e.}$.
If we substitute $f$ with $\eta$ in $G_{\alpha}^{\eta}f=G_{\alpha}\left(f-\eta G_{\alpha}^{\eta}f\right)$,
and $\alpha\downarrow0$, then, 
\[
R\eta\left(1-R^{\eta}\eta\right)=R\left(\eta-\eta R^{\eta}\eta\right)=R^{\eta}\eta\leq1,\ \mathsf{m}\textrm{-a.e.}
\]
Using recurrence, we have $R^{\eta}\eta=1$ from, 
\[
^{\exists}g\in L_{+}^{1}\left(\mathbb{R}^{d};\mathsf{m}\right),\ Rg=\infty,\ \mathsf{m}\textrm{-a.e.}
\]
This is because, since $E=\left\{ x\in E\left|Rg\left(x\right)=\infty\right.\right\} $,
$\mathsf{m}\textrm{-a.e.}$ and $R\eta\left(1-R^{\eta}\eta\right)\leq1$,
$\mathsf{m}\textrm{-a.e.}$, when $C=\left\{ x\in E\left|R\eta<\infty\right.\right\} $,
we have 
\[
R\eta=0,\ \textrm{and}\ \eta=0,
\]
$\mathsf{m}\textrm{-a.e.}$ on $C$. However, this contradicts $\eta>0,\ \mathsf{m}\textrm{-a.e.}$.
The rest is when $R\eta=\infty$, however, $R\eta\left(1-R^{\eta}\eta\right)\leq1$
only holds when $R^{\eta}\eta=1$. If we let$f_{n}=G_{\frac{1}{n}}^{\eta}\eta$,
we have $0\leq f_{n}\uparrow1$ and $n\rightarrow\infty$, then 
\[
\mathcal{E}\left(f_{n},f_{n}\right)\leq\mathcal{E}_{\frac{1}{n}}\left(f_{n},f_{n}\right)=\left\langle \eta-\eta f_{n},f_{n}\right\rangle _{\mathsf{m}}\leq\int_{\mathbb{R}^{d}}\left(\eta-\eta f_{n}\right)d\mathsf{m}\rightarrow0.
\]
\end{proof}

\section{Proof of Lemma \ref{lem:BGX-Cor1} }

The result of \cite{Brown-EdGeorge-Xu_08} states that when $X\sim\mathcal{N}_{d}\left(\mu,v_{x}I\right)$
and $Y\sim\mathcal{N}_{d}\left(\mu,v_{y}I\right)$, the predictive
density $\hat{p}^{\pi}\left(\left.y\right|x\right)=\frac{1}{M^{\pi}\left(x;v_{x}I,0\right)}\int_{\mathbb{R}^{d}}p_{v_{y}I,0}\left(\left.y\right|\mu\right)p_{v_{x}I,0}\left(\left.x\right|\mu\right)\pi\left(\mu\right)d\mu$
of $Y$ after observing data, $X=x$, evaluated by the KL loss, 
\[
L_{\mathsf{KL}}\left(\mu,\hat{p}^{\pi}\left(\cdot\left|x\right.\right)\right)=\mathsf{KL}\left(p_{v_{y}I,0}\middle\|\hat{p}^{\pi}\right)=\int_{\mathbb{R}^{d}}\log\frac{p_{v_{y}I,0}\left(y\left|\mu\right.\right)}{\hat{p}^{\pi}\left(y\left|x\right.\right)}p_{v_{y}I,0}\left(y\left|\mu\right.\right)dy,
\]
the Bayes risk difference between the predictive density, $\hat{p}^{\pi_{U}}$,
under uniform prior, $\pi_{U}\equiv1$, and the predictive density,
$\hat{p}^{\pi}$, under proper prior, $\pi$ is 
\begin{equation}
B_{\mathsf{KL}}\left(\mu,\hat{p}^{\pi_{U}}\right)-B_{\mathsf{KL}}\left(\mu,\hat{p}^{\pi}\right)=\frac{1}{2}\int_{v_{w}}^{v_{x}}\frac{1}{v^{2}}\left[B_{Q}^{v}\left(\mu,\hat{\mu}_{\textrm{MLE}}\right)-B_{Q}^{v}\left(\mu,\hat{\mu}_{\pi}\right)\right]dv\label{eq:BGX-08_1}
\end{equation}
where, 
\[
B_{Q}^{v}\left(\mu,\hat{\mu}\right)=\int_{\mathbb{R}^{d}}\mathbb{E}\left[\left\Vert \hat{\mu}-\mu\right\Vert ^{2}\right]\pi\left(\mu\right)d\mu,\ \ \hat{\mu}_{\pi}=\int_{\mathbb{R}^{d}}\mu\pi\left(\left.\mu\right|x\right)d\mu,\ v_{w}=\frac{v_{x}v_{y}}{v_{x}+v_{y}}.
\]
The $\frac{1}{v^{2}}\left[B_{Q}^{v}\left(\mu,\hat{\mu}_{\textrm{MLE}}\right)-B_{Q}^{v}\left(\mu,\hat{\mu}_{\pi}\right)\right]$
inside the integral on the right-hand side in eq.~\eqref{eq:BGX-08_1}
is four times the Dirichlet form of the standard Brownian motion $4\mathcal{E}_{\textrm{BM}}\left(\sqrt{M^{\pi}},\sqrt{M^{\pi}}\right)$
\citep[eq.~(1.3.4)]{Brown_71} (the special case of $v=1$ is
Brown's identity): 
\begin{alignat}{1}
\frac{1}{v^{2}}\left[B_{Q}^{v}\left(\pi,\hat{\mu}_{\textrm{MLE}}\right)-B_{Q}^{v}\left(\pi,\hat{\mu}_{\pi}\right)\right] & =4\int_{\mathbb{R}^{d}}\left\Vert \nabla_{z}\sqrt{M^{\pi}\left(z;vI,0\right)}\right\Vert ^{2}dz=4\mathcal{E}_{\textrm{BM}}\left(\sqrt{M^{\pi}},\sqrt{M^{\pi}}\right).\label{eq:Brown-71}
\end{alignat}
Here, $M^{\pi}\left(z;vI,0\right)$ is the marginal likelihood of
$p_{v}\left(\left.z\right|\mu\right)=\mathcal{N}\left(\mu,vI\right)$
under the prior, $\pi\left(\mu\right)$ 
Finally, we have 
\[
B_{\mathsf{KL}}\left(\mu,\hat{p}^{\pi_{U}}\right)-B_{\mathsf{KL}}\left(\mu,\hat{p}^{\pi}\right)=2\int_{v_{w}}^{v_{x}}\left[\int_{\mathbb{R}^{d}}\left\Vert \nabla_{z}\sqrt{M^{\pi}\left(z;vI,0\right)}\right\Vert ^{2}dz\right]dv.
\]

\section{Proof of Theorem \ref{thm:A-harmonic}}

\begin{proof} 1. Let $P$ be the Markov operator corresponding to
the transition probability $\hat{p}^{\rho}(y|x)$, and define $\mathcal{G}=I-P$.
We define a test function $u_{R}(x)$ for $r=\|x\|$ as follows: 
\[
u_{R}(x)=v_{R}(r)=\begin{cases}
1 & (0\le r\le1)\\
{\displaystyle \frac{\int_{r}^{R}\frac{1}{M^{\rho}\left(s;A,\nu,\gamma\right)s^{d-\alpha+1}}ds}{\int_{1}^{R}\frac{1}{M^{\rho}\left(s;A,\nu,\gamma\right)s^{d-\alpha+1}}ds}} & (1<r<R)\\
0 & (r\ge R)
\end{cases}
\]
Clearly, it satisfies $v_{R}(1)=1$ and $v_{R}(R)=0$. We evaluate
the spatial scale of the operator $\mathcal{G}$ acting on the test function
$u_{R}$: 
\[
\mathcal{G}u_{R}(x)=(I-P)u_{R}(x)=\int_{\mathbb{R}^{d}}(u_{R}(x)-u_{R}(z))\hat{p}^{\rho}(z|x)dz
\]
By the assumption $\hat{p}^{\rho}(z|x)\sim\|x-z\|^{-(d+\alpha)}$,
the following holds: 
\[
\mathcal{G}u_{R}(x)=O(R^{-\alpha})
\]
Using this, we evaluate the $k=2$ term of the Dirichlet form. Since
$\mathcal{G}$ is a self-adjoint operator, we have 
\[
\langle \mathcal{G}^{2}u_{R},u_{R}\rangle_{M^{\rho}}=\langle \mathcal{G}u_{R},\mathcal{G}u_{R}\rangle_{M^{\rho}}=\|\mathcal{G}u_{R}\|_{L^{2}(M^{\rho})}^{2}
\]
Because the magnitude of $\mathcal{G}u_{R}(x)$ is $O(R^{-\alpha})$, its squared
norm is of order $O(R^{-2\alpha})$. By repeating the same logic,
for any $k\ge2$, we obtain 
\[
\langle \mathcal{G}^{k}u_{R},u_{R}\rangle_{M^{\rho}}=O(R^{-k\alpha})
\]
Since $\alpha>0$, in the limit as $R\to\infty$, the higher-order
terms from $O(R^{-2\alpha})$ onwards decay rapidly compared to the
first term $\mathcal{E}_{\rho,1}(u_{R},u_{R})\asymp O(R^{-\alpha})$
and can be asymptotically neglected completely: 
\[
\mathcal{E}_{\rho}(u_{R},u_{R})=\mathcal{E}_{\rho,1}(u_{R},u_{R})\times(1+O(R^{-\alpha}))
\]
Therefore, the divergence or convergence condition of the Dirichlet
form for the test function is completely dominated by the behavior
of the first term $\mathcal{E}_{\rho,1}$. Thus, for the purpose of
analyzing the boundary between recurrence and transience (i.e., the
relationship between $d$ and $\alpha$), we lose no generality by
using only the first term for the analysis.

2. We evaluate the first term $\mathcal{E}_{\rho,1}$ of the non-local
(integral-type) Dirichlet form with respect to the smoothly decaying
radial test function $u_{R}$. Differentiating the test function $u_{R}$
with respect to $r$ yields 
\[
\frac{dv_{R}(r)}{dr}=-\frac{1}{M^{\rho}\left(r;A,\nu,\gamma\right)r^{d-\alpha+1}}\times\left(\int_{1}^{R}\frac{1}{M^{\rho}\left(s;A,\nu,\gamma\right)s^{d-\alpha+1}}ds\right)^{-1}
\]
Substituting this into the scale evaluation for the radial integral
$\mathcal{E}_{\rho,1}(u_{R},u_{R})\asymp\int_{1}^{R}(v_{R}')^{2}M^{\rho}\left(r;A,\nu,\gamma\right)r^{d-\alpha+1}dr$,
we obtain 
\[
\mathcal{E}_{\rho,1}(u_{R},u_{R})\asymp\left(\int_{1}^{R}\frac{1}{M^{\rho}\left(s;A,\nu,\gamma\right)s^{d-\alpha+1}}ds\right)^{-1}
\]
Therefore, the condition for the process to be recurrent (i.e., the
capacity of the Dirichlet form converging to $0$ in the limit) immediately
reduces to 
\[
\int_{1}^{\infty}\frac{1}{M^{\rho}\left(r;A,\nu,\gamma\right)r^{d-\alpha+1}}dr=\infty.
\]

3. We now prove the divergence of the integral based on the decay
rate $\beta$ of the prior distribution. When the prior distribution
satisfies $\rho(\theta)\sim\|\theta\|^{-\beta}\ (\beta\le d)$, we
evaluate the asymptotic behavior of the marginal likelihood $M^{\rho}\left(r;A,\nu,\gamma\right)$:
\[
M^{\rho}\left(x;A,\nu,\gamma\right) =\int_{\mathbb{R}^{d}}p_{A,\nu,\gamma}(x|\theta)\rho(\theta)d\theta\approx\int_{\mathbb{R}^{d}}\frac{1}{\|x-\theta\|^{d+\alpha}}\frac{1}{\|\theta\|^{\beta}}d\theta
\]
This integral is a convolution of two heavy-tailed distributions.
Sufficiently far from the origin ($r=\|x\|\to\infty$), the decay
order of the convolution is dominated by the ``heavier tail" (the
slower decay). By the assumption $\beta\le d<d+\alpha$, the dominant
tail is that of $\rho(\theta)$. Hence, the decay of the marginal
likelihood is given by 
\[
M^{\rho}\left(r;A,\nu,\gamma\right)\sim\frac{1}{r^{\beta}}=r^{-\beta}
\]
Substituting this into the integral condition obtained in Step 2,
we have 
\[
\int_{1}^{\infty}\frac{1}{r^{-\beta}r^{d-\alpha+1}}dr=\int_{1}^{\infty}r^{\beta-d+\alpha-1}dr
\]
A necessary and sufficient condition for an integral of the form $\int_{1}^{\infty}r^{p}dr$
to diverge is that the exponent satisfies $p\ge-1$. 
\[
\beta-d+\alpha-1\ge-1\quad\iff\quad\beta-d+\alpha\ge0\quad\iff\quad\beta\ge d-\alpha
\]
This matches the condition $d-\alpha\le\beta\le d$ that
the improper prior must satisfy. This completes the proof of the sufficient
condition for the theorem. \end{proof}
\end{document}